\documentclass[12pt]{amsart}
\usepackage[colorlinks=true,citecolor=green,linkcolor=magenta]{hyperref}
\usepackage{amsmath}
\usepackage{verbatim}
\usepackage{cleveref}
\usepackage{amsfonts}
\usepackage{amssymb}
\usepackage{blkarray}
\usepackage{color,colortbl}
\usepackage[all,cmtip]{xy}  
\usepackage{enumerate}
\usepackage[top=1in, bottom=1in, left=1in, right=1in]{geometry}
\usepackage{mathrsfs}

\usepackage{stmaryrd}
\usepackage{bbold}
\renewcommand{\k}{\mathbb{k}}
\usepackage{cleveref}
\usepackage{soul}
\usepackage{arydshln}
\usepackage{tikz-cd}
\usepackage{graphicx}
\theoremstyle{definition}
\newtheorem{theorem}{Theorem}[section]

\numberwithin{equation}{section}

\newtheorem{cor}[theorem]{Corollary}
\newtheorem{lemma}[theorem]{Lemma}
\newtheorem{proposition}[theorem]{Proposition}

\newtheorem{notation}[theorem]{Notation}
\newtheorem{observation}[theorem]{Observation}

\theoremstyle{definition}
\newtheorem{definition}[theorem]{Definition}
\newtheorem{setting}[theorem]{Setting}
\newtheorem{example}[theorem]{Example}

\newtheoremstyle{TheoremNum}
        {8pt}{8pt}              %%% space between body and theorem
        {\upshape}                      %%% theorem body font
        {}                              %%% Indent amount (empty = no indent)
        {\bfseries}                     %%% theorem head font
        {.}                             %%% Punctuation after theorem head
        {.5em}                             %%% Space after theorem head
        {\theoremname{#1}\theoremnote{ \bfseries #3}}%%% theorem head spec
  \theoremstyle{TheoremNum}

\newcommand{\m}{\mathfrak{m}}

\newcommand{\lcm}{\operatorname{lcm}}

\newcommand{\spc}{\operatorname{Spec}}

\newcommand{\inn}{\operatorname{in}}

\newcommand{\kk}{\mathbf{k}}
\newcommand{\rees}{\mathcal{R}}

\newcommand{\Supp}{\operatorname{Supp}}

\newcommand{\Ass}{\operatorname{Ass}}

\newcommand{\Min}{\operatorname{Min}}

\newcommand{\hgt}{\operatorname{ht}}	
\newcommand{\grade}{\operatorname{grade}}

\newcommand{\Sym}{\operatorname{Sym}}

\DeclareMathOperator{\fitt}{Fitt}

\newcommand{\Gs}[1]{\rm (G_#1)}

\renewcommand{\leq}{\leqslant}
\renewcommand{\geq}{\geqslant}

\newcommand{\p}{\mathfrak{p}}

\newcommand{\ann}{\operatorname{ann}}

\renewcommand{\a}{\mathfrak{a}}

\title{Rees algebra and almost linearly presented ideals in three variables}

\address{Department of Mathematics, Indian Institute of Technology Delhi}
\author[Kumar]{Suraj Kumar}
\email{Suraj.Kumar@maths.iitd.ac.in}

\begin{document}
\maketitle
\begin{abstract}
    Let $R=\k[x,y,z]$ and $I=(f_0,\dots,f_{n-1})$ be a height two perfect ideal which is almost linearly presented (that is, all but the last column have linear entries, but the last column has entries which are homogeneous of degree $2$). Further we suppose that after modulo an ideal generated by two variables, the presentation matrix has rank one. Also, the ideal $I$ satisfies $\Gs{{2}}$ but not $\Gs{3}$, then we obtain explicit formulas for the defining ideal of the Rees algebra $\rees(I)$ of $I$. 
\end{abstract}
\section{Introduction}
The Rees algebra of an ideal is a fundamental object in commutative algebra and algebraic geometry, capturing both algebraic and geometric information about the ideal and its associated blow-up. Let $R$ be a Noetherian ring and $I \subset R$ a finitely generated ideal. The Rees algebra of $I$, denoted $\rees(I)$, is defined as
$$\rees(I)=R[It]=R[f_1t,\dots,f_nt] \subset R[t],$$
where $f_1,\dots,f_n$ generates $I$. Since $I$ is finitely generated, the Rees algebra can be viewed as the image of the surjective graded homomorphism
$$\Psi: R[w_1,\dots,w_n] \rightarrow \rees(I)$$
where $\Psi(w_i)=f_it, 1 \leq i \leq n$. The kernel $\mathcal{A} :=\ker \Psi$, known as the defining ideal of the Rees algebra, is a central object of study.

A compelling connection to computational geometry arises when considering the moving curve ideal (see \cite{Cox08}), first studied in the context of computer-aided geometric design. Suppose $R=\k[x_1,x_2]$, and let $a,b,c \in R$ be homogeneous polynomials of the same degree with $\gcd(a,b,c)=1$. The rational map
$$(x_1:x_2) \rightarrow (a(x_1,x_2):b(x_1,x_2):c(x_1,x_2))$$
parametrizes a rational plane curve in $P^2$. A moving curve of degree $m$ is a polynomial in $R[x,y,z]$ of the form
$$\sum_{i+j+k=m} A_{ijk}(x_1,x_2) x^i y^j z^k,$$
which follows the parametrization if 
$$\sum_{i+j+k=m} A_{ijk}(x_1,x_2) a(x_1,x_2)^i b(x_1,x_2)^j c(x_1,x_2)^k = 0.$$
The collection of all such polynomials forms an ideal in $R[x,y,z]$, known as the moving curve ideal, denoted $MC$. Remarkably, this ideal coincides with the defining ideal $\mathcal{A}$
of the Rees algebra $\rees(I)$, thereby bridging computational and theoretical perspectives.

In general, when $R=\k[x_1,\dots,x_d]$ is a polynomial ring over an algebraically closed field $\k$, and the generators $f_i$ are homogeneous of same degree, the rational map
$$\psi: P_{\k}^{d-1} \rightarrow P_{\k}^{n-1}, (x_1:\dots:x_d) \rightarrow (f_1:\dots:f_n)$$
defines a projective subvariety via its closed image. Finding the defining equations of this image is the classical implicitization problem (\cite{CortadellasD'Andrea14,CoxGoldmanZhang00,Cox08}), and the Rees algebra serves as the bihomogeneous coordinate ring of the graph of $\psi$. From this perspective, the defining ideal $\mathcal{A}$ captures the implicit equations of the image of $\psi$.

From the algebraic standpoint, the symmetric algebra $\Sym(I) \cong \frac{R[w_1,\dots,w_n]}{\mathcal{L}}$, provides an approximation of the Rees algebra, where $\mathcal{L}=([w_1\cdots w_n] \cdot \varphi)$, with $\varphi$ being a presentation matrix of $I$. Since $\Psi(\mathcal{L})=0$, we have $\mathcal{L} \subseteq \mathcal{A}$, and hence the Rees algebra fits into the short exact sequence
$$0 \rightarrow \frac{\mathcal{A}}{\mathcal{L}} \rightarrow \Sym(I) \rightarrow \rees(I) \rightarrow 0.$$
When $\mathcal{A}=\mathcal{L}$, we say that $I$ is of linear type, and the defining equation the Rees algebra are completely determined by $\varphi$. Characterizing when ideals are of linear type has been a central question (see for example \cite{micali,huneke_dsequences,Huneke82}), with foundational results
by Herzog-Simis-Vasconcelos \cite{HSV83}. One of the condition in \cite{HSV83} is,
\begin{align}\label{G infinity}
\mu(I_\p)\leq\hgt \p\text{ for all }\p\in V(I).
\end{align}
A generalization of this, due to Artin-Nagata \cite{ArtinNagata72}, is the $\Gs{s}$ condition
\begin{align*}
    \Gs{{s}}: \mu(I_\p)\leq \hgt \p\text{ for all }\p\in V(I)\text{ and }\hgt\p\leq s-1.
\end{align*}
 Of course, \eqref{G infinity} is equivalent to saying that $\Gs{{s}}$ is satisfied for every $s$. Ideals satisfying the $\Gs{d}$ condition and possessing favorable homological properties- such as being strongly Cohen-Macaulay (see \cite{Huneke83} for definition) have well behaved Rees algebras (for example see {\cite{UlrichVasconcelos93,Morey-Ulrich,BoswellMukundan16}}). Despite significant progress, the structure of $\mathcal{A}$ remains elusive in many cases where $I$ is not of linear type. Some of the generators of $\mathcal{A}$ can be obtained from the Jacobian dual of $\varphi$, a matrix $B(\varphi)$ satisfying the identity
$$[w_1\cdots w_n] \cdot \varphi = [x_1 \cdots x_d] \cdot B(\varphi),$$
and minors of $B(\varphi)$ often yield additional generators of $\mathcal{A} \backslash {\mathcal{L}}$.
After restricting the  presentation matrix $\varphi$ of $I$ to be almost linearly, that is, all but last column of $\varphi$ are linear and last column consist of homogeneous entries of arbitrary degree $m\geq 2$ in \cite{Morey-Ulrich} shows that the explicit characterization of generators of $\mathcal{A}$ other than $\mathcal{L}$ leads to the concept of iterated Jacobian dual and those extra generators of $\mathcal{A}$ is exactly the maximal minors of iterated Jacobian dual(\cite{BoswellMukundan16}).
Since then, there are many other classes of ideals for which the defining ideal of the Rees algebra is known (see \cite{Buse09,CortadellasD'Andrea14,CoxHoffmanWang,HongSimisVasconcelos,KPU11,KPU17,Weaver23,KimMukundan20,SurajMukundan24,SurajMukundan25}).

The height two perfect ideals satisfying the $\Gs{{d-1}}$ condition has also been investigated in \cite{Lan14,Lan17,DoriaRamosSimis18,ConstantiniPriceWeaver23}. Even in these cases the minors of Jacobian dual matrix appear as part of generators of $\mathcal{A}$. But the Jacobian dual in these cases, is written with respect to $(x_1,\dots,x_dw_1)$ instead of $(x_1,\dots,x_d)$.

In this article, we explore a significant and largely unexplored class of following ideals. The corresponding defining equation are computed explicitly.

\begin{setting}\label{Main Setting} Let $R=\k[x,y,z]$ be a polynomial ring over algebraically closed field $\k$, with the homogeneous maximal ideal $\m=(x,y,z)$. Let $I$ be an $R$ ideal satisfying:
\begin{itemize}
    \item $I$ is height two perfect ideal;
    \item $I=(f_1,\dots,f_n)$ is generated by homogeneous elements of same degree in $R$;
    \item The presentation matrix $\varphi$ of $I$ is $n \times (n-1)$, almost linearly presented, meaning all but the last column are linear and entries in the last column are homogeneous of degree 2;
    \item $I_1(\varphi)=\m$, and $\mu(I)=n >4$;
    \item after a change of variables, the matrix $\varphi$ has rank at most 1 modulo the prime ideal generated by any two variable;
    \item $I$ satisfies the condition $\Gs{2}$, but not $\Gs{3}$.
\end{itemize}
\end{setting}
Leveraging the rank condition on $\varphi$ modulo primes generated by two variables, we classify such matrices into three distinct types- Case I, Case II, and Case III-each leading to different structure for the defining ideal $\mathcal{A}$. For each case, we compute the generators of $\mathcal{A}$ explicitly, thereby shedding light on the implicit equations governing the associated rational map and contributing to the effort of understanding Rees algebra of almost linearly presented ideals.

And our main result is as follows.

\begin{theorem}({\Cref{another characterization of defining ideal of Rees algebra}, \Cref{equality of D' and A in case iii}, \Cref{equality of D and A in case iii}}).
Let $I$ be ideal in \Cref{Main Setting}. Consider the ring $S=\k[\underline{x}, \underline{w}]$.
\begin{enumerate}
    \item Let $B=S/{\mathcal{J}}$, where $\mathcal{J}=(l_1,\dots,l_{n-2}):(x,y,zw_0)$ is the Cohen-Macaulay prime ideal.
    \begin{enumerate}
        \item If the presentation matrix $\varphi$ is in Case I, then in the ring $B$,
        \begin{align*}
            \overline{\mathcal{A}} \cong \frac{\overline{l_{n-1}}\overline{\mathcal{K}}^{(2)}}{\overline{z^2w_0^2}},
        \end{align*}
        where $\overline{\mathcal{K}}=(\overline{zw_0}):\overline{(x,y,zw_0)}$, is height one Cohen-Macaulay ideal of the ring $B$.
        \item If the presentation matrix $\varphi$ is in Case II, then in the ring $B$,
        \begin{align*}
         \overline{\mathcal{A}} \cong \frac{\overline{l_{n-1}}\mathcal{K}'}{\overline{z^2w_0}},   
        \end{align*}
      where $\mathcal{K}'=(\overline{z^2w_0}):\overline{(x,y,zw_0)}^2$, is height one ideal of $B$.  
    \end{enumerate}
    \item If the presentation matrix $\varphi$ is in Case III, define the ring $B=S/{\mathcal{J}}$, where $\mathcal{J}=(l_1,\dots,l_{n-2}):(x,y)$ is the Cohen-Macaulay prime ideal. Then in the ring $B$,
    \begin{align*}
        \overline{\mathcal{A}} \cong \frac{\overline{l_{n-1}}\overline{\mathcal{K}}''}{\overline{z^2w_0}},
    \end{align*}
    where $\overline{\mathcal{K}}''=(\overline{z^2w_0}):\overline{(x,y,z^2w_0)}$, is height one Cohen-Macaulay ideal of the ring $B$.
\end{enumerate}
\end{theorem}

We compute $\overline{\mathcal{K}}^{(2)}$ and $\mathcal{K}'$ using the Grobner basis technique. It has a complete description in \Cref{w_0 is regular over ideal}, \Cref{subcase two}, \Cref{subcase three}, \Cref{subcase four} and \Cref{subcase five}. Also, the description of $\overline{\mathcal{K}}''$ is in \Cref{K in case III}.

The paper is organized as follows. In \Cref{section 2 and tools}, we present all the tools required in the sequel. In \Cref{section 3 and almost linearly presented matrix}, we present different forms of the presentation matrix (\Cref{presentation matrix}) and also show that minimal primes of $I_2(\varphi)$ is unique (\Cref{minimal prime is unique}), and prove the general form of defining ideal $\mathcal{A}$ of Rees algebra (\Cref{colon ideal related to defining ideal of Rees algebra}). In \Cref{section 4 and rees algebra in case i}, we show the defining ideal of the Rees algebra $\mathcal{A} \cong \overline{\mathcal{K}}^{(2)}$ for Case I in \cref{another characterization of defining ideal of Rees algebra} and we also give the explicit generators of $\overline{\mathcal{K}}^{(2)}$, using the Grobner basis technique, in \Cref{w_0 is regular over ideal}, \Cref{subcase two}, \Cref{subcase three} and \Cref{subcase four}. In \Cref{section 5 and rees algebra in Case ii}, we compute the explicit generators of $\mathcal{K}'$, using the Grobner basis technique in \Cref{subcase five}. We also show the defining ideal of Rees algebra $\overline{\mathcal{A}} \cong \mathcal{K}'$ for Case II in \Cref{equality of D' and A in case iii}. In \Cref{section 6 and rees algebra in case iii}, we give the complete description of $\overline{\mathcal{K}}''$ in \Cref{K in case III}. We also prove the defining ideal of Rees algebra $\overline{\mathcal{A}} \cong \overline{\mathcal{K}}''$ in \Cref{equality of D and A in case iii}. In \Cref{section 7 and examples}, we present the examples \Cref{example 1}, \Cref{example 2} and \Cref{example 3}.

\section{Preliminaries}\label{section 2 and tools}
\subsection{Fitting Ideals and $\Gs{s}$ condition}
\begin{definition}
Let $R$ be a ring and $I$ be an $R$-ideal. Let $s\in \mathbb{N}$, an ideal $I$ is said to satisfy the $\Gs{s}$ condition if $\mu(I_Q) \leq \hgt{Q}$ for all $Q \in V(I)$ with $\hgt{Q} \leq s-1$. 
\end{definition}
We also say that an ideal $I$ is said to satisfy $\Gs{\infty}$ if $I$ satisfies $\Gs{s}$ for every $s \in \mathbb{N}$.

Often we need a more practical definition of the $\Gs{s}$ condition. One such condition is the following: If $\varphi$ is a presentation matrix of the ideal $I$, then the condition $\Gs{s}$ is related to the height of the ideal of minors of $\varphi$.

\begin{definition}
 Let $\varphi$ be $m \times n$ matrix with entries in $R$ and $M$ be the cokernel of the linear map $\varphi : R^n \rightarrow R^m$. The $i^{th}$ Fitting ideal of $M$, 
 \begin{align*}
\fitt_i(M)= I_{m-i}(\varphi), 0 \leq i \leq \min\{m,n\},
 \end{align*}
 where $I_i(\varphi)$ is ideal generated by  $i \times i $ minors of $\varphi$. Also, $I_i(\varphi)= 0, i > \min\{m,n\}$ and $I_i(\varphi)=R, i \leq 0$. It is well known that the Fitting ideal depends only on the module $M$ and is independent of the presentation matrix $\varphi$.   
\end{definition}
Since the $\Gs{s}$ condition is a question of checking the number of generators of the ideal locally at a prime ideal, the following result is useful to check when it is not satisfied.
\begin{proposition}\cite[Proposition 20.6]{eisenbud_book}\label{NASC on fitting ideal} When $R$ is local, we have $\fitt_i(M)=R$ if and only if $\mu(M) \leq i$. In fact, $V(\fitt_i(M))=\{\p \in \spc(R) ;  \mu(M_\p) > i\}$.
\end{proposition}

\begin{lemma}\cite[Corollary 2.2]{Lan14}\cite[Remark 1.3.2]{Vasconcelosbook}\label{Gs and Fitting ideals} The ideal $I$ satisfies $\Gs{s}$ if and only if $\hgt{\fitt_i(I)} \geq i+1$ for all $i \leq s-1$.
\end{lemma}

\subsection{Rees Algebra of Ideals}
To compute the defining ideal of the Rees algebra $\mathcal{R}(I)$, we define the map
\begin{align*}
    \Phi:R[w_0,\dots, w_d]&\rightarrow \mathcal{R}(I)=R[It]\\
    w_i&\rightarrow f_it
\end{align*}
The kernel $\ker\Phi$ of the above map is the called the \textit{defining ideal} of the Rees algebra $\mathcal{R}(I)$. We set $\mathcal{A}=\ker\Phi$.
Often to compute the defining ideal of the Rees algebra we approximate it using the symmetric algebra $\Sym(I)$ which has the presentation 
\begin{align*}
    \Sym(I)\cong\frac{R[w_0,\dots,w_d]}{\mathcal{L}}
\end{align*}
where $\mathcal{L}=I_1([w_0\cdots w_d]\circ\varphi)$. Clearly $\mathcal{L}\subseteq\mathcal{A}$ and hence the map $\Phi$ factors through the symmetric algebra $\Sym(I)$.
\begin{align*}
    \xymatrix{
    R[w_0,\dots,w_d] \ar[rd]\ar[rr]^-\Phi& & \mathcal{R}(I)\\
    &\Sym(I)\ar[ru]&
    }
\end{align*}

We define a bi-degree on $R[w_0,\dots,w_d]$, that is, $\deg x_i=(1,0)$ and $\deg w_j=(0,1)$. Thus, every element of the defining ideal of the Rees algebra has bi-degree $(a,b)$. The first component is referred to as the $\underline{x}$-degree while the second component measures the $\underline{w}$-degree.
\subsection{Jacobian dual}
\begin{definition}
    Let $R=\kk[x_1,\dots,x_d]$ is ring of polynomial over the field $\kk$ and ideal $I$ of $R$ has $m \times n$ presentation matrix whose entries are  forms in $R$. Assume that $I_1(\varphi)=\langle a_1,\dots,a_r\rangle$. Then there exist a $r \times n$ matrix $B(\varphi)$ with entries in $\kk[x_1,\dots,x_d,w_0,\dots ,w_{m-1}]$ so that
    \begin{align*}
        [w_0\cdots w_{m-1}] \cdot \varphi = [a_1\cdots a_r] \cdot B(\varphi).
    \end{align*} The matrix $B(\varphi)$ is called \textit{Jacobian dual matrix} of $\varphi$.
\end{definition}

The minors of the Jacobian dual matrix are equations that are in $\mathcal{A}$, but not in $\mathcal{L}$. For example, when the presentation matrix $\varphi$ is linear, that is, $I_1(\varphi)=( x_1,\dots,x_d)$, then we can show that $I_1(\varphi)I_d(B(\varphi))\subseteq \mathcal{L}\subseteq \mathcal{A}$, a prime ideal in $S=R[w_0,\dots,w_{m-1}]$. Since $x_i\not\in \mathcal{A}$ for all $1\leq i\leq d$, we have $I_d(B(\varphi))\subseteq \mathcal{A}$.

\subsection{Genric Residual Intersections}
Given a nonzero ideal $I$ in a Noetherian ring $R$. If either $I=R$ or $s$ is any positive integer such that $I$ satisfies $\Gs{s}$ where $s \geq \max{1, \hgt I}$, its generic $s$-residual intersection was defined by Huneke and Ulrich in [8]as follows. Assume that $f_1,\dots,f_n$ is a generating set of $I$. Let $X$ be a generic $n \times s$ matrix and let $Y=R[X]$. Set $(a_1,\dots,a_s) =(f_1\dots,f_n) \cdot X$.

\begin{definition}\cite[1.3]{HunekeUlrich1988}
 The ideal $I_s:=(a_1,\dots,a_s)X:I$ is called the generic $s$-residual intersection of $I$ with respect to $f_1\dots,f_n$.   
\end{definition}

When the ideal $I$ is a complete intersection, its generic residual intersections can be computed explicitly.
\begin{theorem}\cite[1.8]{HunekeUlrich1988}\label{computation of generic residual intersection}
 Let $R$ be a local Cohen–Macaulay ring. Let $f_1,\dots,f_n$ be a regular sequence in $R$. Then $I_s=(a_1,\dots,a_s,I_n(X))$.  
\end{theorem} 

\begin{theorem}\cite[1.4]{HunekeUlrich1988}\label{generic residual intersection is Cohen Macaulay}
    Let $R$ be a local Cohen–Macaulay ring and $I$ a strongly Cohen–Macaulay ideal satisfying the condition $\Gs{s}$, where $s\geq g=\grade I \geq 1$. Then a generic $s$-residual intersection $I_s$ of $I$ is a Cohen–Macaulay ideal and, if $I$ satisfies $\Gs{{s+1}}$, it is a geometric $s$-residual intersection of $IT$.
\end{theorem}

\begin{theorem}\cite[1.5]{HunekeUlrich1988}\label{deformation}
 Let $R$ be a local Gorenstein ring and $I$ a strongly Cohen–Macaulay ideal satisfying the condition $\Gs{{s+1}}$, where $s \geq g=\grade I \geq 1$. Let $I_s$ be a generic $s$-residual intersection of $I$, and let $J$ be an arbitrary $s$-residual intersection of $I$ in $R$. Then there exists $P\in Spec(T)$ such that $(T_P, (I_s)_P)$ is a deformation of $(R, J)$, That is, $(R,J) =(T_P/(\underline{x}), ((I_s)_P, \underline{x})/(\underline{x}))$ for some sequence $\underline{x}$ in $T_P$ which is regular on $T_P$ and on $T_P/(I_s)_P$.   
\end{theorem}

\subsection{Matrix pencils over $\k[x,y]$}
In this subsection we collect the results on linear matrices with entries in $\k[x,y]$ which we would need in the sequel. Any matrix $M$ with entries $\k[x,y]$ can be written in the form $xA+yB$ where $A,B$ are matrices of the same dimensions as $M$ with entries in $\k$. Two of these matrices $xA+yB$ and $xA'+yB'$ are said to be \textit{strictly equivalent} if there exists invertible matrices $C,D$ such that $xA'+yB'=C(xA+yB)D$. For extensive study on these matrices and their possible forms we refer to \cite{Gantmacher_matrices}. The following result characterizes the canonical forms of the linearly presented matrices $M$.

We will  write $[M,N]$ to represent the block diagonal form $\begin{bmatrix}
    M &0\\
    0 & N
\end{bmatrix}$.
\begin{theorem}{(\cite[Lemma 1]{Matrix_pencils},Kronecker)} \label{matrix pencils}A linearly presented matrix $M$ is strictly equivalent to the matrix $[O,L_{n_1},\dots,L_{n_r},L_{m_1}^t,\dots,L_{m_s}^t,N_1,\dots,N_t]$
\begin{comment}
\begin{equation}
\resizebox{0.35\hsize}{!}{$
    \begin{bmatrix}
        O&&&&&&&&&\\
        &L_{n_1}&&&&&&&&\\
        &&\ddots&&&&&&&\\
        &&&L_{n_r}&&&&&&\\
        &&&&L_{m_1}^t&&&&&\\
        &&&&&\ddots&&&&\\
        &&&&&&L_{m_s}^t&&&\\
        &&&&&&&N_1&&\\
        &&&&&&&&\ddots&\\
        &&&&&&&&&N_t
    \end{bmatrix}
    $}
\end{equation}
\end{comment}
where $O$ is the zero matrix, $L_n$ is a $n\times (n+1)$ matrix of the form $\begin{bmatrix}
    x&y&0 &\cdots &0\\
    0 & x & y &\cdots &0\\
    \vdots & \vdots &\vdots &\ddots &\vdots\\
    0 &0 &\cdots&\cdots x&y
\end{bmatrix}$ and $N_m$ is a square $m\times m$ matrix of the form $\begin{bmatrix}
    \alpha x+\beta y& \alpha'x+\beta'y&0 &\cdots &0\\
    0 &\alpha x+\beta y&\alpha'x+\beta'y&0\cdots &0\\
    \vdots & \vdots &\vdots &\ddots &\vdots \\
    0 &0&0 &\cdots &\alpha x+\beta y
\end{bmatrix}$ with $\alpha,\beta\in\k$.
\end{theorem}

\section{Almost Linearly Presented Ideals}\label{section 3 and almost linearly presented matrix}
%\textbf{\textit{Setting}} $\colon$

\subsection{Presentation matrix $\varphi$ of the ideal $I$}
First we will give canonical forms of the presentation matrix $\varphi$. Since $I$ is the perfect ideal of height two, therefore $I$ is generated by the maximal minors of $\varphi$ (\cite[Theorem 20.6]{BrunsHerzogbook}). Hence, performing $\k$-linear combination of the rows and column will not change the ideal $I$.

\begin{lemma}\label{presentation matrix}
    Let $\varphi$ be the presentation matrix as given in \Cref{Main Setting}. After a possible change of variables and elementary row and column operations, $\varphi$ has three different forms
    \begin{align*}
\varphi = \begin{pmatrix}
*&* \cdots * &z^2\\
& &*'\\
&\varphi'&\vdots\\
& &*'\\
\end{pmatrix},&& \varphi = \begin{pmatrix}
*&* \cdots z &*'\\
& &*'\\
&\varphi'&\vdots\\
& &*'\\
\end{pmatrix}&&\text{and}&&
\varphi = \begin{pmatrix}
*&* \cdots z &z^2+*'\\
& &*'\\
&\varphi'&\vdots\\
& &*'\\
\end{pmatrix},
\end{align*}
where $\varphi'$ is a $(n-1) \times (n-2)$ matrix. The entries of $\varphi'$ and the $*$ entries are linear in $\k[x,y]$, while the  $*'$ entries are homogeneous of degree $2$ in $\k[x,y,z]$ but do not contain the term $z^2$.
\end{lemma}
\begin{proof}
    We know that after the change of variables, the matrix $\varphi$ has rank at most 1 modulo ideal generated by two variables (see condition 5 in \Cref{Main Setting}). Suppose $\p=(x,y)$ be the required ideal and $\varphi_1$ be the matrix with entries in $\k[z]$, obtained from $\varphi$ after reducing the entries of $\varphi$ modulo $\p$. That is, $\varphi_1=z \varphi_2$, where $\varphi_2$ is $n \times (n-1)$ matrix over $\k[z]$. Now the rank of $\varphi_2$ over the field $\k(z)$ can not be 0 because $I_1(\varphi)=\m=(x,y,z)$ therefore the rank of $\varphi_2$ over the field $\k(z)$ will be exactly 1. But the last column of $\varphi$ is of degree 2, so we wish to retain this information, thus after applying the row or column operation $\varphi_2$ has possible forms over $\k[z]$.

\begin{align*}
\begin{pmatrix}
0&0 \cdots 1 &0\\
& &0\\
&0&\vdots\\
& &0\\
\end{pmatrix},&& \begin{pmatrix}
0&0 \cdots 0 &z\\
& &0\\
&0&\vdots\\
& &0\\
\end{pmatrix}&&\text{and} &&\begin{pmatrix}
0&0 \cdots 1 &z\\
& &0\\
&0&\vdots\\
& &0\\
\end{pmatrix}.
\end{align*}
Equivalently, $\varphi_1$ will be
\begin{align*}
\begin{pmatrix}0&0 \cdots z &0\\
& &0\\
&0&\vdots\\
& &0\\
\end{pmatrix},&&\begin{pmatrix}
0&0 \cdots 0 &z^2\\
& &0\\
&0&\vdots\\
& &0\\
\end{pmatrix}&&\text{and} && \begin{pmatrix}
0&0 \cdots z &z^2\\
& &0\\
&0&\vdots\\
& &0\\
\end{pmatrix}.  
\end{align*}
Thus $\varphi$ has three possible forms 
\begin{align*}
\begin{pmatrix}
*&* \cdots z &*'\\
& &*'\\
&\varphi'&\vdots\\
& &*'\\
\end{pmatrix},&& \begin{pmatrix}*&* \cdots * &z^2\\
& &*'\\
&\varphi'&\vdots\\
& &*'\\
\end{pmatrix} &&\text{and } && \begin{pmatrix}
*&* \cdots z &z^2+*'\\
& &*'\\
&\varphi'&\vdots\\
& &*'\\
\end{pmatrix},    
\end{align*}
where $\varphi'$ is a $(n-1) \times (n-2)$ matrix, and the entries of $\varphi'$ and the $*$ entries are linear in $ \k[x,y]$ and the entries $*'$ are homogeneous of degree $2$ which do not contain the term $z^2$.
\end{proof}

From \Cref{presentation matrix}, one of the possible form of presentation matrix is
\begin{align*}
\varphi= \begin{pmatrix}
*&* \cdots z &z^2+*'\\
& &*'\\
&\varphi'&\vdots\\
& &*'\\
\end{pmatrix}    
\end{align*}
After applying the column operation which does not change the ideal $I$, the last column of $\varphi$ will become $[*''~\cdots~*'']^{t}$ where $*''$ is homogeneous polynomial of degree 2 in $\k[x,y,z]$ which do not contain the term $z^2$. Since the last column is $[*''~\cdots~*'']^t$ which represents $(n-1)$-th symmetric equation $l_{n-1}$, that is, $*''w_0+\cdots+w_{n-1}*''=l_{n-1} \notin (x,y,zw_0)^2$. Hence, third form of $\varphi$ is 
\begin{align*}
\varphi= \begin{pmatrix}
*&* \cdots z &*''\\
& &*''\\
&\varphi'&\vdots\\
& &*''\\
\end{pmatrix}    
\end{align*}
where $\varphi'$ is $(n-1) \times (n-2)$ linear matrix over $\k[x,y]$. The entries $*$ are linear in $\k[x,y]$ and the entries $*''$ are homogeneous polynomial of degree two over $\k[x,y,z]$ which do not contain the term $z^2$ such that $*''w_0+\cdots+*''w_{n-1}=l_{n-1}\notin (x,y,zw_0)^2$.

Now, for the first form of $\varphi$ in \Cref{presentation matrix}, the last column is $[*' \dots *']^{t}$. Thus, $*'w_0+\cdots+*'w_{n-1}=l_{n-1} \in (x,y,zw_0)^2$ or $*'w_0+\cdots+*'w_{n-1}=l_{n-1} \notin (x,y,zw_0)^2$. But $*'w_0+\cdots+*'w_{n-1}=l_{n-1} \notin (x,y,zw_0)^2$ case fall in above case. So, for the first form of $\varphi$ in \Cref{presentation matrix}, we consider $*'w_0+\cdots+*'w_{n-1}=l_{n-1}\in (x,y,zw_0)^2$.

Since there are three choices for the presentation matrix $\varphi$ of the ideal $I$, henceforth we follow the following notation
\begin{notation}\label{case i and ii presentation matrix}
Following the previous theorem, we associated three cases to the presentation matrices
\begin{enumerate}
    \item Case I:
    $\begin{pmatrix}
*&* \cdots z &*'\\
& &*'\\
&\varphi'&\vdots\\
& &*'\\
\end{pmatrix}$, $l_{n-1} \in (x,y,zw_0)^2$.
\item Case II: $\begin{pmatrix}
*&* \cdots z &*''\\
& &*''\\
&\varphi'&\vdots\\
& &*''\\
\end{pmatrix}$, $l_{n-1} \notin (x,y,zw_0)^2$.
\item Case III: $\begin{pmatrix}
*&* \cdots * &z^2\\
& &*'\\
&\varphi'&\vdots\\
& &*'\\
\end{pmatrix}$.
\end{enumerate}
The entries of $\varphi'$ and the $*$ entries are linear in $ \k[x,y]$, while the  $*'$ and $*''$ entries are homogeneous of degree $2$ but do not contain the term $z^2$. 

Henceforth, we associate Case I, Case II and Case III to the presentation matrices described above.
\end{notation}

%From \Cref{presentation matrix }, $\varphi'$ is $(n-1) \times (n-2)$ sub matrix of $\varphi$ with linear entries in the polynomial ring $\kk[x,y]$. Thus $\varphi'=xA+yB$, where $A$ and $B$ are matrices with entries in the field $\kk$. Now using Kronecker theorem, $\varphi'=xA+yB$ is strictly equivalent to $[O,L_{m_1},\cdots,L_{m_a},L_{n_1}^t,\cdots,L_{n_b}^t,N_1,\cdots,N_c]$, where $O$ is a matrix with zero entries. Again using {\color{red}[Lan, Lemma 3.3] and [Lan, Lemma 3.4]}, we have $\varphi'=xA+yB=[O,L_m^t,N_1,\cdots,N_c]$. Therefore we have 
%\begin{enumerate}
    %\item Case I: $\begin{pmatrix}
%*&* \cdots * &z^2\\
%& &*'\\
%&\varphi'&\vdots\\
%& &*'\\
%\end{pmatrix}$
%\item Case II: %$\begin{pmatrix}
%*&* \cdots z &*'\\
%& &*'\\
%&\varphi'&\vdots\\
%& &*'\\
%\end{pmatrix}$
%\end{enumerate}
%where $\varphi'=[O,L_{m}^t,N_1,\cdots,N_c]$ and the $*$ entries are linear in $ \kk[x,y]$, while the  $*'$ entries are homogeneous of degree $2$ but do not contain the term $z^2$.

\begin{proposition}
{\cite[Proposition 3.2]{Lan14}}\label{Lan locally linear type}
Let $I$ and $\varphi$ be as in 
 the above setting. If $\p \in V(I) \backslash V(I_{n-2}(\varphi))$ then $I_\p$ is of linear type.    
\end{proposition}

\begin{proposition}\label{minimal prime is unique}
 With the assumption of above setting, the set $\Min(I_{n-2}(\varphi))$ has only one prime.   
\end{proposition}
\begin{proof}
Since $I$ satisfies $\Gs{2}$ but does not satisfy $\Gs{3}$, $\hgt I_{n-1}(\varphi)=\hgt I_{n-2}(\varphi)=2$ by \Cref{Gs and Fitting ideals}.
Let $\widetilde{\varphi}$ be the submatrix of $\varphi$ containing the first $n-2$ columns. From the presentation matrix $\varphi$ in \Cref{case i and ii presentation matrix}, it is clear that $I_{n-1}(\varphi) \subseteq I_{n-2}(\widetilde{\varphi}) \subseteq (x,y)$ and hence $\hgt I_{n-2}(\widetilde{\varphi}) =2$ since $\hgt I_{n-1}(\varphi)=2$. Therefore, $\Min(I_{n-2}(\widetilde{\varphi}))=\{(x,y)\}$. But $I_{n-2}(\widetilde{\varphi}) \subseteq I_{n-2}(\varphi) \subseteq (x,y)$, hence we get $\Min(I_{n-2}(\varphi))=\{(x,y)\}$.    
\end{proof}

The following theorem gives the structure of the defining ideal of the Rees algbera under the action of $\Gs{{2}}$  condition. Although there is general version of this description \cite[Proposition 3.9]{SurajMukundan24}. 
\begin{proposition}\label{colon ideal related to defining ideal of Rees algebra}\cite{SurajMukundan24}
With assumption in above setting,  then $\mathcal{A}= \mathcal{L} : (x,y)^{\infty}$.    
\end{proposition}
\begin{proof}
    Let $f \in \mathcal{L} : (x,y)^{\infty}$ then $f \in \mathcal{L} : (x,y)^t ; t \in \mathbb{N}$. Thus $x^tf, y^tf \in \mathcal{L} \subseteq \mathcal{A}$. But $\mathcal{A}$ is prime so $f \in  \mathcal{A}$. Thus $\mathcal{L} : (x,y)^{\infty} \subseteq \mathcal{A}$. Now \Cref{Lan locally linear type} implies that 
 $\Supp(\mathcal{A}/ \mathcal{L})\supset  V(I_{n-2}(\varphi))$. Since $\Supp(\mathcal{A}/ \mathcal{L})=V(\ann(\mathcal{A}/ \mathcal{L}))$ so we have $V(\ann(\mathcal{A}/ \mathcal{L})) \supset V(I_{n-2}(\varphi))$. Using \Cref{minimal prime is unique}, we know that $\p=(x,y)\in\Min(I_{n-2}(\varphi))$ is unique. Since $\mathcal{L} : (x,y)^{\infty} \subseteq \mathcal{A}$, we have $\p\in \Supp(\mathcal{A}/\mathcal{L})$. Thus, $\Supp(\mathcal{A}/\mathcal{L})=V(\ann(\mathcal{A}/\mathcal{L}))=V(I_{n-2}(\varphi))$. Again, $\sqrt{I_{n-2}(\varphi)}= \p = (x,y)$. Thus, $\p = (x,y)= \sqrt{\ann(\mathcal{A}/ \mathcal{L})}$. Now, $x,y \in \sqrt{\ann(\mathcal{A}/ \mathcal{L})}$ implies that there exists $t_1,t_2 \in \mathbb{N}$ such that $x^{t_1},y^{t_2} \in \ann(\mathcal{A}/ \mathcal{L})$. Fix $t = \max\{t_1,t_2\}$ then $(x,y)^t \subseteq \ann(\mathcal{A}/ \mathcal{L})$. Therefore, $(x,y)^t \mathcal{A}  \subseteq \mathcal{L}$. Hence $\mathcal{A}= \mathcal{L} : (x,y)^{\infty}$
\end{proof}

\section{Description of Defining ideal of Rees algebra when presentation matrix is in Case I}\label{section 4 and rees algebra in case i}

With the assumption of \Cref{Main Setting}, and that the presentation matrix $\varphi$ is of Case I (see \Cref{case i and ii presentation matrix}). In this section we will first give a better description of the Rees algebra which is comparable to \Cref{colon ideal related to defining ideal of Rees algebra} which is proved in \Cref{another characterization of defining ideal of Rees algebra}. Then we compute this colon ideal using methods in \cite{BoswellMukundan16, KPU11}, where we identify the the defining equations of the Rees algebra to the symbolic power of a height one unmixed ideal in a Cohen-Macaulay ring. Using Grobner basis technique, we give complete characterization of symbolic power of this height one unmixed ideal.

\begin{notation}\label{notation for case I}
Consider the $S$-ideal, $(x,y,zw_0)$ then $[w_0~\cdots~w_{n-1}] \cdot \varphi = [x~y~zw_0] \cdot B(\varphi)$. Here, $B(\varphi)$ is the Jacobian dual of the presentation matrix $\varphi$ with respect to ideal $(x,y,zw_0)$. Now onward, we will follow this definition of Jacobian dual. 

Recall that the ideal defining the symmetric algebra is $\mathcal{L} =(l_1,\dots,l_{n-1})$.  Since $l_{n-2} \in \mathcal{L}$, from the presentation matrix $\varphi$ (see \Cref{case i and ii presentation matrix}) of $I$ we have $l_{n-2}=zw_0 + xg_1 + yg_2$ where $g_i$'s are linear polynomial in $w_j; 1 \leq j \leq n-1$.
\end{notation}

\begin{lemma}\label{ second colon ideal related to defining ideal of Rees algebra}\cite[Lemma 5.1]{SurajMukundan24}
 With the assumption of \cref{Main Setting}, $\mathcal{A}=\mathcal{L}:(x,y,zw_0)^{\infty}$.   
\end{lemma}

\begin{observation}\label{colon ideal in second case}
 Assume the \Cref{Main Setting}, then $$\mathcal{L} : (x,y,zw_0) =\mathcal{L} + I_3(B(\varphi)).$$    
\end{observation}

\begin{proof}
Let $\mathcal{U}=(x,y,zw_0)$, then $\mathcal{U}$ is a complete intersection. We show that $\mathcal{L}: \mathcal{U}$ is the $(n-1)$-residual intersection of $\mathcal{U}$ then \Cref{deformation}, \Cref{computation of generic residual intersection} and \Cref{generic residual intersection is Cohen Macaulay} imply $\mathcal{L}:\mathcal{U}=\mathcal{L}+I_3(B(\varphi))$. It suffices to show that $\hgt (\mathcal{L}:\mathcal{U}) \geq n-1$. 

Let $P$ be a minimal prime of $\mathcal{L}: \mathcal{U}$. Suppose $p=P \bigcap R$. Now, $(x,y) \subseteq p$ or $(x,y) \nsubseteq p$. If $(x,y) \subseteq p$ then $(x,y)+I_3(B(\varphi)) \subseteq P$. Since, $I_3(B(\varphi))$ is prime ideal of height $n-3$, therefore $\hgt((x,y)+I_3(B(\varphi))) \geq n-1$. Thus, $\hgt P \geq n-1$. If $(x,y) \nsubseteq p$, then $(\mathcal{L}:\mathcal{U})_p \subseteq \mathcal{A}_p$. But \Cref{Lan locally linear type}, \Cref{minimal prime is unique} implies  $\mathcal{L}_p=\mathcal{A}_p$. Thus, $(\mathcal{L}:\mathcal{U})_p=\mathcal{A}_p$. This implies $\mathcal{A} \subseteq P$ and hence $\hgt P \geq n-1$. 
\end{proof}

One can in fact give the exact saturation degree of the symmetric algebra with respect to $(x,y,zw_0)$. This is in fact dependent only on the degree of the entries in the last column.

\begin{theorem}\cite{SurajMukundan24}\label{Defining ideal in form of saturation of an ideal}
    Assume the \cref{Main Setting}, then one has  $\mathcal{A}=\mathcal{L}:(x,y,zw_0)^m$, where $m \geq 2$.
\end{theorem}

\begin{notation}\label{notation 5.1}
Consider the ideal $\mathcal{U}=(x,y,zw_0)$ and define $\mathcal{J}= (l_1,l_2,\dots,l_{n-2}): \mathcal{U}$ in the ring $S=R[\underline{x}, \underline{w}]$. Let $B$ be the ring $S/{\mathcal{J}}$. Let $f \in S$ then $\overline{f}$ denote the image in ring $B$.
Let $\varphi''$ 
 be $n \times (n-2)$ 
 matrix obtained by deleting the last column of the matrix $\varphi$. Let $B(\varphi'')$ be the Jacobian dual matrix of $\varphi''$.
 
Let $A=(a_{ij})$ be the $2 \times (n-2)$ matrix obtained by deleting the last row from $B(\varphi'')$. Define the $S$-ideal 
\begin{align*}
\mathcal{K}=(l_1,\dots,l_{n-2}) + I_{2}(A) + (zw_0).    
\end{align*}

We fix $(n-2)$-th column of $A$ and  for $i=1,2,\cdots,{n-3}$, we define $2\times 2$ sub matrix $A_i$ of $A$ such that first column of $A_i$ is the $i$-th column of $A$ and second column of $A_i$ the $(n-2)$-th column of $A$. Denote $\det(A_i)=\a_i$. Also, for $1 \leq i < j \leq n-3$, let $C_{ij}$ be the $2 \times 2$ submatrix of $B(\varphi'')$ such that the first and second columns of $C_{ij}$ are the $i$-th and $j$-th columns of $B(\varphi'')$. Denote $c_{ij}=\det(C_{ij})$.
Clearly $I_2(A)=(\a_1,\a_2,\cdots,\a_{n-3})+(\{c_{ij} ; 1\leq i <j \leq n-3 \})$.

Let $A_{i_x}$ is the matrix obtained from $A_i$ after replacing the first row of $A_i$ by the row vector $[0~1]$, similarly we can define $A_{i_y}$. Denote $\a_{i_x}=\det(A_{i_x})$ and $\a_{i_y}=\det(A_{i_y})$. Clearly $a_{i_x}$ and $a_{i_y}$ are linear polynomials in $w_j$'s.

Also, $l_{n-2}=zw_0+ xg_1+yg_2$, where $g_1,g_2$ are linear polynomials in $w_j ; 1 \leq j \leq n-1$. Set $ \widetilde{l_{n-2}}=l_{n-2}-zw_0$ where $l_{n-2} \in \mathcal{L}$ is the $(n-2)$-th symmetric equation. 
\end{notation} 

\begin{lemma}\label{colon ideal of K}
    The ring $B$ is a Cohen-Macaulay domain of dimension $5$. And the $B$-ideal, $\overline{\mathcal{K}}$ is Cohen-Macaulay ideal of height one. We also have the equalities of $B$-ideals,
    \begin{align*}
\overline{\mathcal{K}}^{(j)} = \overline{(zw_0)}^{(j)} : \overline{(x,y,zw_0)}^{(j)} \\ \overline{(x,y,zw_0)}^{(j)} = \overline{(zw_0)}^{(j)} : \overline{\mathcal{K}}^{(j)} ; 1 \leq j \leq 2.
    \end{align*}
\end{lemma}
\begin{proof}
   Proof is similar to the proof \cite[Observation 5.6, Observation 5.7, Lemma 5.10]{SurajMukundan24}.
\end{proof}

Now define the $B$-ideal $\mathcal{D} =\frac {\overline{l_{n-1}}\overline{\mathcal{K}}^{(2)}}{\overline{z^2}\overline{w_0^2}}$. Since $l_{n-1} \in (x,y,zw_0)^2 \subseteq (x,y,zw_0)^{(2)}$, by \Cref{colon ideal of K}, $\overline{l_{n-1}}\overline{\mathcal{K}}^{(2)} \subseteq \overline{(zw_0)}^2 $. Thus the $B$-ideal, $\mathcal{D}$ is well defined.
Also $\overline{l_{n-1}} \in \overline{\mathcal{A}}$ so $\overline{z^2w_0^2}\mathcal{D} \subseteq \overline{\mathcal{A}}$. But $\overline{\mathcal{A}}$ is prime and $\overline{zw_0} \notin \overline{\mathcal{A}}$ therefore  $\mathcal{D} \subseteq \overline{\mathcal{A}}$.

\begin{theorem}\label{another characterization of defining ideal of Rees algebra}
    Assume the \Cref{Main Setting}, the $B$-ideals $\mathcal{D}$ and $\mathcal{\overline{A}}$ are equal.
\end{theorem}

\begin{proof}
Since $B$ is Cohen-Macaulay and $\overline{\mathcal{K}}^{(2)}$ is unmixed of height one. Also $\mathcal{D} \subseteq \overline{\mathcal{A}}$ this implies that $\mathcal{D}$ is the proper ideal of $B$, and $\overline{K}^{(2)} \cong \mathcal{D}$ so we have $\mathcal{D}$ is also unmixed of the height one by Serre conditions. Thus we need to show that $\mathcal{D}$ and $\overline{\mathcal{A}}$ are equal locally at height one associated primes.

Also, $\overline{\mathcal{A}}$ is height one prime ideal and $\overline{(x,y,z)} \nsubseteq \overline{\mathcal{A}}$ and $\overline{(x,y,w_0)} \nsubseteq \overline{\mathcal{A}}$. Let $P \in Ass(\mathcal{D})$, if $ \overline{(x,y,z)} \subseteq P$ or $\overline{(x,y,w_0)} \subseteq P$, then $\overline{\mathcal{K}}\nsubseteq P$. Thus
\begin{align*}
    \overline{(x,y,zw_0)}^{(2)}_{P} =\overline{(zw_0)}^{2}_{P} : \overline{\mathcal{K}}^{(2)}_{P}= \overline{(zw_0)}^{2}_{P} : B_{P}= \overline{(zw_0)}^{2}_{P}.
\end{align*}
Also,  $\overline{\mathcal{A}}_{P} = \overline{(l_{n-1})}_{P} : \overline{(x,y,zw_0)}^{(2)}_P$, that is, $B_{P}= \overline{(l_{n-1})}_{P} : \overline{(x,y,zw_0)}^{(2)}_{P}$, this will imply $\overline{(x,y,zw_0)}^{(2)}_{P} \subseteq \overline{(l_{n-1})}_{P}$. But $\overline{(l_{n-1})}_{P} \subseteq \overline{(x,y,zw_0)}^{(2)}_{P}$ so $\overline{(l_{n-1})}_{P} = \overline{(x,y,zw_0)}^{(2)}_{P}$. Therefore, $\overline{(x,y,zw_0)}^{(2)}_{P}= \overline{(zw_0)}^{2}_{P} = \overline{(l_{n-1})}_{P}$. Thus, $\mathcal{D}_{P} =\frac {\overline{(l_{n-1})}_{P} \overline{\mathcal{K}}^{(2)}_{P}} {\overline{z^2}\overline{w_0^2}} = \frac {\overline{(zw_0)}^{2}_{P} \overline{\mathcal{K}}^{(2)}_{P}}{\overline{z^2}\overline{w_0^2}}= B_{P} = \overline{\mathcal{A}}_{P}$.

If $\overline{(x,y,z)} \nsubseteq P$ and $\overline{(x,y,w_0)} \nsubseteq P$ then $I_{P}$ will be of linear type (\Cref{Lan locally linear type}), so $\overline{\mathcal{A}}_{P} = \overline{(l_{n-1})}_{P}$ but $\overline{(l_{n-1})}_{P} \subseteq \mathcal{D}_{P}$ therefore $\overline{\mathcal{A}}_{P}= \mathcal{D}_{P}$. Hence $\overline{\mathcal{A}}_{P} = \mathcal{D}_{P}$, for all $P\in \Ass(\mathcal{D})$.
\end{proof}

Now we have $\overline{\mathcal{K}}^2 \subseteq \overline{\mathcal{K}}^{(2)}$. If $\overline{\mathcal{K}}^2=\overline{\mathcal{K}}^{(2)}$ then we can explicitly give the generators of $\overline{\mathcal{A}}$. Therefore, we are interested for the case where $\overline{\mathcal{K}}^2 \neq \overline{\mathcal{K}}^{(2)}$. From \Cref{notation 5.1}, $B(\varphi'')$ is $3 \times (n-2)$ matrix, consider the ideal $I_2(B(\varphi''))$. Now, either $w_0 \in I_2(B(\varphi''))$ or $w_0 \notin I_2(B(\varphi''))$. Thus, depending on this condition, we will explicitly characterize the $B$-ideal $\overline{\mathcal{K}}^{(2)}$.

To compute the symbolic power, we use the description $\overline{\mathcal{K}}^{(2)}=\overline{\mathcal{K}}^2:(s)^\infty$, where $s$ is a non-zero divisor on $\overline{\mathcal{K}}$ and $s$ is in every embedded associated prime of $\overline{\mathcal{K}}^2$.

\begin{lemma}\label{xi is nzd on K}
Consider the $B$-ideals $(\overline{x})$. Then in the ring $B$,  $\overline{\mathcal{K}} :(\overline{x}) = \overline{\mathcal{K}}$.
\end{lemma}
\begin{proof}
 Notice that $\overline{\mathcal{K}}$ is Cohen-Macaulay and all associated primes of $\overline{\mathcal{K}}$ are the associated primes of $(\overline{zw_0})$ but not in the associated prime of $\overline{(x,y)}$. Thus, after possibly changing coordinates, we can choose either $\overline{x}$ or $\overline{y}$ which is not in any associated primes of $\overline{\mathcal{K}}$ either (by prime avoidance). So take $\overline{x}$, then the result follows.
\end{proof}

From Kronecker theorem (\Cref{matrix pencils}), $\varphi'$ is strictly equivalent to the block diagonal matrix \begin{align*}[O~L_{p_1}~\cdots~L_{p_m}~L_{s_1}'~\cdots~L_{s_t}'~M_{r_1}~\cdots~M_{r_k}].
\end{align*}

Now, we state some of the results from \cite{Lan14} on $\varphi'$. These results give us more restrictions on $\varphi'$.

\begin{lemma}\cite[Lemma 3.3]{Lan17}\label{O block not appear}
    The zero block $O$ and the $L_m$ block do not appear in the above decomposition of $\varphi'$.
\end{lemma}

\begin{lemma}\cite[Lemma 3.4]{Lan17}\label{only one bolck of L'}
 There is one and only one block $L'_n$ in the decomposition of $\varphi'$.   
\end{lemma}
 Thus, \Cref{O block not appear}, \Cref{only one bolck of L'} implies that $\varphi'$ will be strictly equivalent to $[L_{n-2}']$ or $[L_{r}~M_{r_1}~\cdots~M_{r_k}]$. Now, depending on the presentation of $\varphi'$, we will explicitly characterize the $B$-ideal $\overline{\mathcal{K}}^{(2)}$.

\subsection{When $\varphi'$ is strictly equivalent to $[L'_{n-2}]$}
Define the $B$-ideals,

$$N_1=(\{\overline{-w_{n-2}\a_{i+1}+w_{n-1}\a_i} ; 1\leq i \leq {n-4}\})$$ and \begin{align*}
    N_2=\overline{\mathcal{K}}^2+(\{\overline{-zw_{n-2}\a_{i+1}+zw_{n-1}\a_i} ; 1\leq i \leq {n-4}\})\\+(\{{\overline{w_{n-1}(\a_{i+1}-\a_i)-w_{n-2}(\a_{i+2}-\a_{i+1})}} ; 1 \leq i \leq n-5\}).
\end{align*} Clearly, $N_i \neq \overline{\mathcal{K}}^2$ for $i=1,2$.

\begin{lemma}\label{x.N_1 is subset of K^2}
    With the assumption of \Cref{Main Setting}. In the ring $B$, we have $\overline{x} \cdot N_i \subseteq \overline{\mathcal{K}}^2$, for $i=1,2$.
\end{lemma}
\begin{proof}
     For $1 \leq j \leq n-3$, consider the $S$-ideal, $(l_j,l_{n-2})$. Then $[l_j,l_{n-2}]=[x~y~zw_0] \cdot B_j$, where $B_j$ is $2 \times 3$ submatrix of $B(\varphi'')$, such that first column of $B_j$ is $j$-th column and second column be the last column of $B(\varphi'')$. Now,  \cite[Lemma 4.3]{BoswellMukundan16}, implies $x\a_i=zw_0 \a_{i_x}$. Similarly, $x \a_{i+1}=zw_0 \a_{{i+1}_x}$. Therefore,
     \begin{align*}
    \overline{x} \cdot (-\overline{w_{n-2}\a_{i+1}}+\overline{w_{n-1} \a_i})&=-\overline{w_{n-2}}\{\overline{x \a_{i+1}}\}+ \overline{w_{n-1}}\{\overline{x \a_i}\} \\&=\overline{w_{n-2}}\{\overline{zw_0 \a_{{i+1}_x}}\}- \overline{w_{n-1}}\{zw_0 \a_{i_{x}}\}
    \\&= \overline{w_{n-2}}\{\overline{zw_0 \a_{{i+1}_x}}\} - \overline{w_{n-1}}\{\overline{zw_0 \a_{{i_x}}}\} \\&=\overline{zw_0}\{\overline{w_{n-2} \a_{{i+1}_x}}-\overline{w_{n-1} \a_{i_x}}\} \in \overline{\mathcal{K}}^2.
\end{align*}
Thus, $\overline{x}\cdot N_1 \subseteq \overline{\mathcal{K}}^2$. That is, ${N_1} \subseteq \overline{\mathcal{K}}^{(2)}$. Similarly, $\overline{x} \cdot N_2 \subseteq \overline{\mathcal{K}}^2$ and hence $N_2 \subseteq \overline{\mathcal{K}}^{(2)}.$
\end{proof}

%{and $N_3=\overline{\mathcal{K}}^2 + (\overline{-zw_{n-1}c_{n-4}+zw_nc_{n-5}}) + (\{\overline{w_{n-1}(\a_{i+1}-\a_{i+2})-w_n(\a_{i+1}-\a_i)} ; 1\leq i \leq {n-3}\})$}.

%\begin{theorem}
   % Assume \Cref{Main Setting}. If $w_1 \in I_2(B(\varphi''))$ then we have \begin{align*}
%\overline{\mathcal{K}}^2 : N_i = \overline{(x,y)}+(\{\overline{a_{j,k}};1\leq j,k \leq n-3\}).
   % \end{align*}
%\end{theorem}
%\begin{proof}
%Since the proof for $i=1,2$ is same so it will be enough to show for just one such $i$. We will show it for $i=1$.    
%\end{proof}

\begin{theorem}\label{w_0 is regular over ideal}
    Assume the assumption of \Cref{Main Setting}. Let $\varphi'$ is strictly equivalent to $[L'_{n-2}]$.
    \begin{enumerate}
        \item If $w_0 \notin I_2(B(\varphi''))$, then we have $\overline{\mathcal{K}}^{(2)}=(\overline{z^2w_0^2})+{N_1}$.
        \item If $w_0 \in I_2(B(\varphi''))$, then we have $\overline{\mathcal{K}}^{(2)}=\overline{\mathcal{K}}^2+ N_2$.
    \end{enumerate}
    \end{theorem}
\begin{proof}
Let the first $n-3$ entries in the first row of the presentation matrix $\varphi$ be $a_1x+b_1y, a_2x+b_2y,\cdots, a_{n-3}x+b_{n-3}y$ where $a_i,b_i \in \k$. Since $\varphi'$ is strictly equivalent to $[L_{n-2}']$. Then Jacobian Dual of $\varphi''$ is 
\begin{align*}
    B(\varphi'')=\begin{pmatrix}
        a_1w_0+w_1&a_2w_0+w_2&\cdots&a_{n-3}w_0+w_{n-3}&w_{n-2}\\
        b_1w_0+w_2&b_2w_0+w_3&\cdots&b_{n-3}w_0+w_{n-2}&w_{n-1}\\
        0&0&\cdots&0&1\\
    \end{pmatrix}
\end{align*}
For $2 \leq i \leq n-3$, $a_iw_0+w_i,b_{i-1}w_0+w_i \in I_2(B(\varphi''))$. Thus, $(a_{i}-b_{i-1})w_0 \in I_2(B(\varphi''))$. But $w_0 \notin I_2(B(\varphi''))$, this will imply $a_i=b_{i-1}$ for $2 \leq i \leq n-3$. Without loss of generality, take $a_i=b_{i-1}=0$ for $2\leq i \leq n-3$ and $a_1=1,b_{n-3}=1$. Thus, the symmetric equations $l_i=xw_i+yw_{i+1}$ for $2 \leq i \leq n-4$, $l_1=x(w_0+w_1)+yw_2, l_{n-3}=xw_{n-3}+y(w_0+w_{n-2})$ and $l_{n-2}=xw_{n-2}+yw_{n-1}+zw_0$.

We have $\overline{x}$ is regular over $\overline{\mathcal{K}}$ and $(\overline{z^2w_0^2})+N_1 \subseteq \overline{\mathcal{K}}^{(2)}$. If $\overline{x}$ is regular over $(\overline{z^2w_0^2})+N_1$ then for all $P \in \Ass{((\overline{z^2w_0^2})+N_1)}$, $\overline{x} \notin P$. Thus, for all $P \in \Ass{((\overline{z^2w_0^2})+N_1)}$,
\begin{align*}
    \{(\overline{z^2w_0^2})+{N_1} \}_P \subseteq (\overline{z^2w_0^2})_P :\overline{(x,y,zw_0)}^{(2)}_P=(\overline{z^2w_0^2})_P:B_P=(\overline{z^2w_0^2})_P.
\end{align*}
Thus, it suffices to show that $\overline{x}$ is regular over $(\overline{z^2w_0^2})+N_1$. Equivalently, to show in the ring $S$,$x$ is regular over $\mathcal{J}+(z^2w_0^2)+N_1$. We will use the Grobner basis technique to prove this.

Let $<$ be the reverse lexicographical monomial order on the ring $S=\k[\underline{x},\underline{w}]$ and 
\begin{align*}
    x<y<z<w_0<\cdots<w_{n-1}.
\end{align*}
Suppose $G=\{h_1,h_2,\dots,h_m\}$ is the Grobner basis of $\mathcal{J}+(z^2w_0^2)+N_1$ then $\mathcal{J}+(z^2w_0^2)+N_1=(h_1,h_2,\dots,h_m)$. Now $x$ will be regular over $\mathcal{J}+(z^2w_0^2)+N_1$ if and only if $\{ \mathcal{J}+(z^2w_0^2)+N_1 \}:(x)=\mathcal{J}+(z^2w_0^2)+N_1$ and to compute the Grobner basis of $\mathcal{J}+\{(z^2w_0^2)+N_1 \} : (x)$ we need to compute the Grobner basis of $\{ \mathcal{J}+(z^2w_0^2)+N_1 \} \bigcap (x)$.

Consider the ring $S[t]$, where $t$ is indeterminate such that $x<y<z<w_0<\cdots<w_{n-1}<t$. Then, we have to compute the Grobner basis of $t{\mathcal{J}}+(tz^2w_0^2)+tN_1+(1-t)x$. For polynomials $f,g$,  $S(f,g)$ be the $S$-pair of polynomials, then we have to compute following $S$-pair of polynomials, $S(th_i,th_j)$ and $S(th_i,(1-t)x)$, $1\leq i <j \leq m$. Observe that, $S(th_i,th_j)=tS(h_i,h_j)=0$ for $1 \leq i < j \leq m$, as $\{h_1,\dots,h_m \}$ is Grobner basis. 

If $x$ does not divide $\inn_{<}h_i$, for all $i$, then $S(th_i,(1-t)x)=xh_i$. Thus the Grobner basis of $\{ \mathcal{J}+(z^2w_0^2)+N_1 \} \bigcap (x)$ will be $\{xh_1,xh_1,\dots,xh_m \}$ and hence, the Grobner basis of $\{ \mathcal{J}+(z^2w_0^2)+N_1 \} :(x)$ will be $G=\{h_1,h_2,\dots,h_m \}$. So, we need to prove that for all $i$, $x$ does not divide $in_{<}h_i$. That is, we have to compute the Grobner basis of $\mathcal{J}+(z^2w_0^2)+N_1$.

Let $G_0=\{l_1,l_2,\dots,l_{n-2} \} \bigcup \{c_{jk} ; 1 \leq j <k \leq n-3 \} \bigcup \{z^2w_0^2, \eta_{i} ; 1 \leq i \leq n-4 \}$ where 
\begin{align*}
    c_{jk}=w_jw_{k+1}-w_{j+1}w_k ; j \neq 1,n-3, k \neq n-3 && in_{<}c_{jk}=w_{j+1}w_{k} \\c_{1k}=w_0w_{k+1}+w_1w_{k+1}-w_2w_k && in_{<}c_{1k}=w_2w_k \\c_{j,n-3}=w_0w_j+w_jw_{n-2}-w_{j+1}w_{n-3} && in_{<}c_{j,n-3}=w_{j+1}w_{n-3} \\ c_{1,n-3}=w_0^2+w_0w_{n-2}+w_0w_1+w_1w_{n-3}-w_2w_{n-3}&& in_{<}c_{1,n-3}=w_2w_{n-3}.
\end{align*}
 And, 
 \begin{align*}
 \eta_i= w_iw_{n-1}^2-2w_{i+1}w_{n-2}w_{n-1}+w_{i+2}w_{n-2}^2 ; 2 \leq i \leq n-5 && in_{<} \eta_i=w_{i+1}w_{n-2}^2 \\ \eta_1=w_0w_{n-1}^2+w_1w_{n-1}^2-2w_2w_{n-2}w_{n-1}+w_3w_{n-2}^2 && in_{1}=w_3w_{n-2}^2
 \\ \eta_{n-4}=w_{n-4}w_{n-1}^2-2w_{n-3}w_{n-2}w_{n-1}+w_0w_{n-2}^2+w_{n-2}^3 && in_{<} \eta_{n-4}=w_{n-2}^3.
\end{align*}
Now, $\mathcal{J}$ is prime ideal, so in stage one iteration of $S$-polynomials, we have to compute following $S$-pair
of polynomials
\begin{align*}
    S(l_{i'},z^2w_0^2) && S(l_{i'},\eta_i) \\ S(c_{jk},z^2w_0^2) &&
    S(c_{jk}, \eta_i) && S(z^2w_0^2,\eta_i).
\end{align*}
For $1 \leq i' \leq n-3$, $S(l_{i'},z^2w_0^2)$ reduces to zero. We have  $in_{<}l_{i'}=yw_{i'+1}$ for $1 \leq i' \leq n-3$ so $2 \leq i'+1 \leq n-2$. Thus, $\gcd(in_{<}l_{i'},z^2w_0^2)=1$. Hence, $S(l_{i'},z^2w_0^2)$ reduces to zero with respect to $l_{i'},z^2w_0^2$ for $1 \leq i' \leq n-3$. Also, $\lcm(in_{<}l_{n-2},z^2w_0^2)=z^2w_0^2$. Thus, $S(l_{n-2},z^2w_0^2)=zxw_0w_{n-2}+yzw_0w_{n-1}$. Applying, division algorithm with respect to $l_{n-2}$, we get $S(l_{n-2},z^2w_0^2)=-x^2w_{n-2}^2-2xyw_{n-2}w_{n-1}-y^2w_{n-1}^2$. Further, reducing with respect to $l_{n-3}$, we get 
\begin{align*}
    S(l_{n-2},z^2w_0^2)=x^2w_{n-2}^2-2x^2w_{n-3}w_{n-1}-2xyw_0w_{n-1}+y^2w_{n-1}^2.
\end{align*}
We have $\mathcal{J}$ a prime ideal and $c_{jk} \in \mathcal{J}$. So, $c_{jk}; 1 \leq j < k \leq n-3$, is irreducible polynomials in $\k[\underline{w}]$, this implies  $\gcd(in_{<}c_{jk},z^2w_0^2)=1$ and thus $S(c_{jk},z^2w_0^2)$ reduces to zero with respect to $c_{jk}, z^2w_0^2$ for $1 \leq j < k \leq n-3$. Also, $w_0$ regular over $I_2(B(\varphi''))$ and $\eta_i \in \k[\underline{w}]$ implies $\gcd(in_{<} \eta_i, z^2w_0^2)=1$. Hence $S(z^2w_0^2, \eta_i)$ reduces to zero with respect to $z^2w_0^2, \eta_i$ for $1 \leq i \leq n-4$.

Now, we investigate $S(l_{i'}, \eta_i); 1 \leq i' \leq n-2, 1 \leq i \leq n-4$. Clearly, for $1 \leq i \leq n-4$, $\gcd(in_{<}l_{n-2}, in_{<} \eta_i)=\gcd(zw_0,w_{i+2}w_{n-2}^2)=1$. Thus, $S(l_{n-2},\eta_i)$ reduces to zero with respect to $l_{n-2}, \eta_i$ for $1 \leq i \leq n-4$. Also, let $2 \leq i \leq n-4$ then $i'=i$ or $i'=i+1$ or $i' \neq i,i+1$. If $i'=i$, then $in_{<} l_i=yw_{n+1}$ and $in_{<} \eta_i=w_{i+2}w_{n-2}^2$. Thus $\gcd(in_{<}l_i, \eta_i)=1$ and hence $S(l_i,\eta_i)$ reduces to zero with respect to $l_i, \eta_i$. If $i'=i+1$ then $\lcm(in_{<} l_{i+1}, in_{<} \eta_i)=\lcm(yw_{i+2}, w_{i+2}w_{n-2}^2)=yw_{i+2}w_{n-2}^2$. Therefore, $S(l_{i+1},\eta_i)=xw_{i+1}w_{n-2}^2-yw_{i}w_{n-1}^2+2yw_{i+1}w_{n-2}w_{n-1}$. Applying the division algorithm with respect to $l_{i-1}, l_{i}$, we get $S(l_{i+1}, \eta_i)=xw_{i+1}w_{n-2}^2+xw_{i-1}w_{n-1}^2-2xw_iw_{n-2}w_{n-1}=x \eta_{i-1}$. Thus, $S(l_{i+1}, \eta_{i})$ reduces to zero. If $i' \neq i, i+1$ then either $i' < i$ or $i'> i+1$, suppose $i' \leq i-1$ then $\gcd(in_{<}l_{i'}, in_{<} \eta_i)= \gcd(yw_{i'+1},w_{i+2}w_{n-2}^2)=1$, as $i'+1 \leq i$ and this implies $S(l_{i'},\eta_{i})$ reduces to zero for $i' \leq i-1$. Similarly, $S(l_{i'}, \eta_i)$ reduces to zero for $i \geq i+1$. Thus, the only case remaining to investigate is $S(l_{i'}, \eta_1)$. Although, for $ 1\leq i' \leq n-2$ but $i' \neq 2$, $S(l_{i'}, \eta_1)$ reduces to zero because $\gcd(in_{<}l_{i'}, \eta_{i})=1$. For $i'=2$, $\gcd(in_{<}l_2, \eta_1) \neq 1$ and hence
\begin{align*}
    S(l_{2}, \eta_1)=xw_2w_{n-2}^2-yw_0w_{n-1}^2-yw_1w_{n-1}^2+2yw_2w_{n-2}w_{n-3}.
\end{align*}
Thus, applying division algorithm with respect to $l_1$, the reduced polynomial is 
\begin{align*}
    S(l_2,\eta_1)=xw_2w_{n-2}^2-2xw_0w_{n-2}w_{n-1}-2xw_1w_{n-2}w_{n-1}-yw_0w_{n-1}^2-yw_1w_{n-1}^2.
\end{align*}
Finally, we have to compute $S(c_{jk}, \eta_i)$. From definition of $\eta_i=w_{n-1} \a_i-w_{n-2} \a_{i+1} ; 1 \leq i \leq n-4$, we have following possibilities $j=i, k=i+1$ (vice-versa) or $j=i,k\neq i+1$ (vice-versa) or $j\neq i, k=i+1$ (vice-versa) or $j \neq i, k \neq i+1$ (vice-versa). If $j=i, k=i+1$, then $in_{<}c_{i,i+1}=w_{i+1}^2$ and $in_{<} \eta_i=w_{i+2}w_{n-2}^2$. Thus $\gcd(in_{<}c_{i,i+1},in_{<} \eta_i)=1$ and this implies $S(c_{i,i+1}, \eta_i)$ reduces to zero with respect to $c_{i,i+1}, \eta_i$. If $j=i,k\neq i+1$, then $k \geq i+2$. This implies $i+2 \leq n-3$, as $k \leq n-3$. So, $in_{<}c_{ik}=w_{i+1}w_k$. If $k>i+2$, $\gcd(in_{<}c_{ik},\eta_i)=1$. And, thus $S(c_{ik},\eta_i)$ reduces to zero. If $k=i+2$, then $\lcm(in_{<}c_{i,i+2}, in_{<} \eta_i)=w_{i+1}w_{i+2}w_{n-2}^2$, $S(c_{i,i+2}, \eta_i)=2w_{i+1}^2w_{n-2}w_{n-1}-w_iw_{i+3}w_{n-2}^2-w_iw_{i+1}w_{n-1}^2$. Applying division algorithm with respect to $\eta_{i+1}, c_{i,i+1}$, $S(c_{i,i+2}, \eta_i)$ reduces to zero. Similarly, if $j\neq i, k=i+1$, then $S(c_{j,i+1},\eta_i)$
reduces to zero. If $j \neq i,k \neq i+1$, then either $j\geq i+1, k\geq i+2$ or $j \leq i-1, k\leq i$. Let $j\geq i+2$ then $k \geq i+3$, which implies that $S(c_{j,k}, \eta_i)$ reduces to zero, as $\gcd(in_{<}c_{jk},in_{<} \eta_i)=1$. If $j=i+1$, then $k \geq i+2$, $\lcm(in_{<}c_{i+1,k},in_{<}\eta_i)=\lcm(w_{i+2}w_k,w_{i+2}w_{n-2}^2)=w_{i+2}w_kw_{n-2}^2$. Thus, $S(c_{i+1,k},\eta_i)=2w_kw_{i+1}w_{n-2}w_{n-1}-w_{i+1}w_{k+1}w_{n-2}^2-w_kw_iw_{n-1}^2=-\eta_{i-1}$, reduces to zero. Similarly, $S(c_{jk},\eta_i)$ reduces to zero for $j \leq i-1, k\leq i$. Hence, $S(c_{jk},\eta_i)$ reduces to zero for $1 \leq j< k \leq n-3, 1 \leq i \leq n-4$.

Therefore, in stage one iteration of Buchberger's algorithm, there are two non-zero $S$-polynomials. That is, 
\begin{align*}
    S(l_{n-2},z^2w_0^2)=x^2w_{n-2}^2-2x^2w_{n-3}w_{n-1}-2xyw_0w_{n-1}+y^2w_{n-1}^2,\\
    S(l_2,\eta_1)=xw_2w_{n-2}^2-2xw_0w_{n-2}w_{n-1}-2xw_1w_{n-2}w_{n-1}-yw_0w_{n-1}^2-yw_1w_{n-1}^2.
\end{align*}
Let $G_1=G_0 \bigcup \{h_{m+1},h_{m+2} \}$ where $h_{m+1}=S(l_{n-2},z^2w_0^2), h_{m+2}=S(l_2,n_1)$ and $in_{<}h_{m+1}=y^2w_{n-1}^2, in_{<}h_{m+2}=yw_1w_{n-1}^2$. Now in stage two iteration of Buchberger's algorithm, we have to compute the following $S$-polynomials
\begin{align*}
    S(l_i',h_{m+1}) && S(c_{jk},h_{m_+1}) && S(z^2w_0^2, h_{m+1})\\ S(\eta_i, h_{m+1})&&
    S(l_i',h_{m+2}) && S(c_{jk},h_{m_+2}) \\ S(z^2w_0^2, h_{m+2}) && S(\eta_i, h_{m+2})&&
    S(h_{m+1},h_{m+2}).
\end{align*}
We show that every $S$-polynomials in stage two will reduces to zero and thus $G_1$ will be required Grobner basis of $\mathcal{J}+(z^2w_0^2)+N_1$. Enough to show $S(l_{i'},h_{m+1}), S(c_{jk},h_{m+1}),\\ S(z^2w_0^2,h_{m+1}), S(\eta_i, h_{m+1}), S(h_{m+1},h_{m+2})$ reduces to zero. And using similar line of argument, one can show $S(l_{i'},h_{m+2}), S(c_{jk},h_{m+2}), S(z^2w_0^2,h_{m+2}), S(\eta_i, h_{m+2})$ also reduces to zero.

We have $in_{<}h_{m+1}=y^2w_{n-1}^2$ and we have already seen that $in_{<}c_{jk}=w_{j+1}w_{k}; 1 \leq j <k \leq n-3$, $in_{<}\eta_i=w_{i+2}w_{n-2}^2; 1 \leq i \leq n-4$. So, $\gcd(in_{<}c_{jk},in_{<}h_{m+1})=\gcd(in_{<}z^2w_0^2,in_{<}h_{m+1})=\gcd(in_{<}\eta_i, in_{<}h_{m+1})=1$. Thus, $S(c_{jk},h_{m+1}),S(z^2w_0^2,h_{m+1}),S(\eta_i,h_{m+1})$ reduces to zero, for $1 \leq i \leq n-4, 1 \leq j<k \leq n-3$.

For $2 \leq i' \leq n-4$, $\lcm(in_{<}l_{i'},  in_{<}h_{m+1})=\lcm(yw_{i'+1}, y^2w_{n-1}^2)=y^2w_{i'+1}w_{n-1}^2$. Therefore,
$S(l_{i'},h_{m+1})=2x^2w_{i'+1}w_{n-3}w_{n-1}-x^2w_{i'+1}w_{n-2}^2-xyw_{i'}w_{n-1}^2+2xyw_0w_{i'+1}w_{n-1}$. Applying the division algorithm with respect to $l_{i'-1}, l_{i'}$, we get $S(l_{i'},h_{m+1})=2x^2w_{i'+1}w_{n-3}w_{n-1}-x^2w_{i'+1}w_{n-2}^2-x^2w_{i'-1}w_{n-1}^2-2x^2w_0w_{i'}w_{n-1}$. Further, reducing it with respect to $c_{i',n-3}$, we get $S(l_{i'},h_{m+1})=2x^2w_{i'}w_{n-2}w_{n-1}-x^2w_{i'+1}w_{n-2}^2-x^2w_{i'-1}w_{n-1}^2=-x^2 \eta_{i'-1}$. Thus, $S(l_{i'},h_{m+1})$ reduces to zero. Similarly, $S(l_1,h_{m+1}), S(l_{n-3},h_{m+1})$ reduces to zero. Also, $\gcd(in_{<}l_{n-2},in_{<}h_{m+1})=\gcd(zw_0,y^2w_{n-1}^2)=1$, so $S(l_{n-2},h_{m+1})$ reduces to zero. Hence, $S(l_{i'},h_{m+1})$ reduces to zero, for $1 \leq i' \leq n-2$.

Now, we will show that $S(h_{m+1},h_{m+2})$ reduces to zero. Since $\lcm(in_{<}h_{m+1},in_{<}h_{m+2})=y^2w_1w_{n-1}^2$ therefore $S(h_{m+1},h_{m+2})=x^2w_1w_{n-2}^2-2x^2w_1w_{n-3}w_{n-1}-2xyw_0w_1w_{n-1}+xyw_2w_{n-2}^2-2xyw_0w_{n-2}w_{n-1}-2xyw_1w_{n-2}w_{n-1}-y^2w_0w_{n-1}$. Applying, division algorithm with respect to $l_1$, we get $S(h_{m+1},h_{m+2})=-2x^2w_1w_{n-3}w_{n-1}-2xyw_0w_1w_{n-1}-2xyw_0w_{n-2}w_{n-1}-2xyw_1w_{n-2}w_{n-1}\\-y^2w_0w_{n-1}-x^2w_0w_{n-2}^2$. Further, reducing it with respect to $l_{n-3}$, $S(h_{m+1},h_{m+2})=-x^2w_0w_{n-2}^2+2x^2w_0w_{n-3}w_{n-1}+2xyw_0^2w_{n-1}-y^2w_0w_{n-1}^2=-w_0h_{m+1}$. Thus, $S(h_{m+1},h_{m+2})$ reduces to zero.

Hence, $G_1=\{ l_1,l_2,\dots,l_{n-2},z^2w_0^2,h_{m+1},h_{m+2}\} \bigcup \{c_{jk}; 1 \leq j < k \leq n-3 \} \bigcup \{ 
\eta_i ; 1 \leq i \leq n-4 \}$ is the Grobner basis of $\mathcal{J}+(z^2w_0^2)+N_1$ and for any $h \in G_1$, $x$ does not divide $in_{<}h$.

$(2)$ Now we suppose $w_0 \in I_2(B(\varphi''))$. Let the first $n-3$ entries in the first row of the presentation matrix $\varphi$ be $a_1x+b_1y, a_2x+b_2y,\cdots, a_{n-3}x+b_{n-3}y$ where $a_i,b_i \in \k$. Since $\varphi'$ is strictly equivalent to $[L_{n-2}']$. Then the Jacobian dual of $\varphi''$ is 
\begin{align*}
   B(\varphi'')= \begin{pmatrix}
        a_1w_0+w_1&a_2w_0+w_2&\cdots&a_{n-3}w_0+w_{n-3}&w_{n-2}\\
        b_1w_0+w_2&b_2w_0+w_3&\cdots&b_{n-3}w_0+w_{n-2}&w_{n-1}\\
        0&0&\cdots&0&1\\
    \end{pmatrix}
\end{align*}
For $2 \leq i \leq n-3$, $a_iw_0+w_i,b_{i-1}w_0+w_i \in I_2(B(\varphi''))$. Thus, $(a_{i}-b_{i-1})w_0 \in I_2(B(\varphi''))$. Since $w_0 \in I_2(B(\varphi''))$, implies $a_i \neq b_{i-1}$ for at least one $i$, where $2 \leq i \leq n-3$. Without loss of generality, take $a_2=0, b_1=1$ and $a_i=b_{i-1}=0$ for $3 \leq i \leq n-3$.

Using the Grobner basis technique and tracing the above proof, we get $\overline{x}$ regular over $\overline{\mathcal{K}}^2+N_2$. Thus, $\overline{x} \notin P$ for all $P \in \Ass{(\overline{\mathcal{K}}^2+N_2)}$. So,
\begin{align*}
    (\overline{\mathcal{K}}^2+N_2)_P \subseteq (\overline{z^2w_0^2})_P: \overline{(x,y,zw_0)}^{(2)}_P=(\overline{z^2w_0^2})_P:B_P=(\overline{z^2w_0^2})_P.
\end{align*}
But, $(\overline{z^2w_0^2})_P \subseteq (\overline{\mathcal{K}}^2+N_2)_P$. Thus, $(\overline{\mathcal{K}}^2+N_2)_P=\overline{\mathcal{K}}^{(2)}_P$. Hence, $\overline{\mathcal{K}}^2+N_2=\overline{\mathcal{K}}^{(2)}$.
\end{proof}

Now, we discuss the case where $\varphi'$ is strictly equivalent to $[L'_{r}~M_{r_1}~\cdots~M_{r_k}]$. It suffices to consider only one sub-block of type $M_{r_i}$, the results will be the same if there is more than one sub-block of type $M_{r_i}$. That is, the description of $\overline{\mathcal{K}}^{(2)}$ in both cases, $\varphi'$ strictly equivalent to $[L'_r~M_{r_1}~\cdots~M_{r_k}]$ and $\varphi'$ strictly equivalent to $[L'_r~M_{r'}]$ are the same.

Since $M_{r'}$ is a square matrix of order $r'$, we can have $r' \geq 2$ or $r'=1$. Depending on $r'$, we will give the complete description of $\overline{\mathcal{K}}^{(2)}$.

\subsection{When $r' \geq 2$} Suppose that the entries in first row of the presentation matrix $\varphi$ are $a_1x+b_1y, a_2x+b_2y, \dots,a_rx+b_ry, a_{r+1}x+b_{r+1}y,\dots,a_{n-3}x+b_{n-3}y$, therefore 
\begin{align*}
    \varphi=\begin{pmatrix}
    a_1x+b_1y & a_2x+b_2y & \cdots & a_rx+b_ry & a_{r+1}x+b_{r+1}y & \cdots a_{n-3}x+b_{n-3}y & z & *'' \\
    & & & & & & &*'\\
& & \varphi' & & & & &\vdots\\
& & & & & & &*'\\
\end{pmatrix},
\end{align*}
where $\varphi$ is strictly equivalent to $[L'_r~M_{r'}]$.

\begin{lemma}\label{the case where w_0 is always zd}
    Assume the assumption of \Cref{Main Setting}. If $\varphi'$ is strictly equivalent to $[L'_r~M_{r'}]$ and $r' \geq 2$, then $w_0 \in I_2(B(\varphi''))$.
\end{lemma}
\begin{proof}
    Suppose $w_0 \notin I_2(B(\varphi''))$ then the Jacobian dual is,
    \begin{align*}
    B(\varphi'')=\begin{pmatrix}
            a_1w_0+w_1 & a_2w_0+w_2 & \cdots & a_rw_0+w_r & a_{r+1}w_0+a^{r_1}w_{r+2} & \cdots & *\\
            b_1w_0+w_2 & b_2w_0+w_3 & \cdots & a_rw_0+w_{r+1} & b_{r+1}w_0+b^{r_1}w_{r+2} &\cdots & *\\
            0 & 0 & \cdots & 0 & 0 & \cdots & *
        \end{pmatrix}.
    \end{align*}
    Now, for $ 2 \leq j \leq r$, $a_jw_0+w_j, b_{j-1}w_0+w_j \in I_2(B(\varphi''))$, thus $(a_j-b_{j-1})w_0 \in I_2(B(\varphi''))$. But $w_0 \notin I_2(B(\varphi''))$, therefore $a_j=b_{j-1}$. From the $(r+1)$-th column of $B(\varphi'')$, we have $a_{r+1}w_0+a^{r_1}w_{r+2}, b_{r+1}w_0+b^{r_1}w_{r+2} \in I_2(B(\varphi''))$. Thus, $a'w_0+w_{r+2}, b'w_0+w_{r+2} \in I_2(B(\varphi''))$, where $a'=a_{r+1}/a^{r_1}$ and $b'=b_{r+1}/b^{r_1}$. Hence, $a'w_0-b'w_0 \in I_2(B(\varphi''))$. Again, $w_0 \notin I_2(B(\varphi''))$ will imply $a'=b'$, that is, $a_{r+1}=a^{r_1}c$ and $b_{r+1}=b^{r_1}c$, for some $c \in \k$ which would imply $r$-th symmetric equation, $l_r=(a^{r_1}x+b^{r_1}y)(cw_0+w_{r+2})$. Now, calculating the $I_{n-1}(\varphi)$ along the $r$-th column, we get $I_{n-1}(\varphi) \subseteq (a^{r_1}x+b^{r_1}y)$, that is, $\hgt I_{n-1}(\varphi) =1$, which is a contradiction as $I_{n-1}(\varphi)=I$, is of height two. Thus, $w_0 \in I_2(B(\varphi''))$. 
\end{proof}

\begin{theorem}\label{subcase two}
    With the assumption of \Cref{Main Setting}. If $\varphi'$ is strictly equivalent to $[L'_r~M_{r'}]$ and $r' \geq 2$, then
        \begin{align*}
        \overline{\mathcal{K}}^{(2)}=\overline{\mathcal{K}}^2 + \overline{(zw_{n-2}\a_{n-3})} \\  +\overline{(dw_{n-2}+bw_{n-1})\a_1-(cw_{n-2}+aw_{n-1})\a_2)} \\ + \overline{(cw_{n-2}+aw_{n-1})\a_3-(dw_{n-2}+bw_{n-1})\a_2+(cw_{n-2}+aw_{n-1})\a_1)} \\ + \overline{(dw_{n-2}+bw_{n-1})\a_{i+1}-(cw_{n-2}+aw_{n-1})\a_i ; 3 \leq i \leq n-5)},
        \end{align*}
        where $a,b,c,d \in \k$ such that the $(n-3)$ symmetric equation $l_{n-3}=zw_0+(aw_{n-1}+cw_{n-2})x+(bw_{n-1}+dw_{n-2}y)$
\end{theorem}
\begin{proof}
The proof is similar to the proof of \Cref{w_0 is regular over ideal}.    
\end{proof}

\subsection{When $r'=1$} 
Since $r'=1$ therefore $M_{r'}=M_1=[ax+by], a,b \in \k$. Define the $B$-ideals, $N_3=(\overline{z^2w_0^2})+(\{\overline{-aw_{n-1}\a_{i+1}+bw_{n-1}\a_i} ; 1\leq i \leq {n-4}\})$ and  \begin{align*}N_4=(\overline{z^2w_0^2})+(\{\overline{-azw_{n-1}\a_{i+1}+bzw_{n-1}\a_i} ; 1\leq i \leq {n-4}\})\\+(\{{\overline{bw_{n-1}(\a_{i+1}-\a_i)-aw_{n-1}(\a_{i+2}-\a_{i+1})}} ; 1 \leq i \leq n-5\}).\end{align*}

Clearly, $N_i \nsubseteq \overline{\mathcal{K}}^2$, for $i=3,4$. Moreover, one can trace the proof of \Cref{x.N_1 is subset of K^2}, and get $\overline{x} \cdot N_i \subseteq \overline{\mathcal{K}}^2$. That is, $N_i \subseteq \overline{\mathcal{K}}^{(2)}$.

\begin{theorem}\label{subcase three}
    Assume the assumption of \Cref{Main Setting}. If $w_0 \notin I_2(B(\varphi''))$, then  we have $\overline{\mathcal{K}}^{(2)}= {N_3}$. \end{theorem}
\begin{proof}
Since $N_3 \subseteq \overline{\mathcal{K}}^{(2)}$.Now one can trace the proof of \Cref{w_0 is regular over ideal} and show that $\overline{x}$ is regular over $N_3$. Thus, for all $P \in \Ass{N_3}$, $\overline{(x,y,zw_0)}^{(2)}_P=B_P$ and hence ${N_3}_P = \overline{\mathcal{K}}^{(2)}_P.$     
\end{proof}

\begin{theorem}\label{subcase four}
    Assume the assumption of \Cref{Main Setting}. If $w_0 \in I_2(B(\varphi''))$, then we have $\overline{\mathcal{K}}^{(2)}= \overline{\mathcal{K}}^2+{N_4}$. \end{theorem}
\begin{proof}
The proof is similar to the proof of \Cref{w_0 is regular over ideal}.    
\end{proof}

\section{Description of Defining ideal of Rees algebra when presentation matrix is in Case II}\label{section 5 and rees algebra in Case ii}

Assume the assumption of \Cref{Main Setting}, and throughout this section, the presentation matrix $\varphi$ is of Case II (see \Cref{case i and ii presentation matrix}). Although the description of the Rees algebra in this case is similar to \Cref{another characterization of defining ideal of Rees algebra} and comparable to \Cref{colon ideal related to defining ideal of Rees algebra}. Again we compute this colon ideal using methods in \cite{SurajMukundan24}, where we identify the the defining equations of the Rees algebra to  a height one unmixed ideal in a Cohen-Macaulay ring.

\begin{notation}
Assume the \Cref{notation for case I}.  
Recall that the defining ideal of symmetric algebra is $\mathcal{L} =(l_1,\dots,l_{n-1})$.  Since $l_{n-1},l_{n-2} \in \mathcal{L}$, then from the presentation matrix $\varphi$ (see \Cref{case i and ii presentation matrix}) of $I$ we have $l_{n-2}=zw_0 + xg_1 + yg_2$ where $g_i$'s are linear polynomial in $w_j; 1 \leq j \leq n-1$ and $l_{n-1}=x^2h_1+y^2h_2+xyh_3+xzh_4+yzh_5$ where $h_i's$ are linear polynomial in $\k[\underline{w}]$ with $h_4,h_5 \notin \k[w_0]$.
\end{notation}

We will collect some of the results of the previous section, which also holds for the presentation matrix $\varphi$ in Case II.

\begin{theorem}\label{ second colon ideal related to defining ideal of Rees algebra}\cite[Lemma 5.1]{SurajMukundan24}
 With the assumption of \cref{Main Setting}, and $\varphi$ in Case II. Then,
 \begin{enumerate}
     \item $\mathcal{A}=\mathcal{L}:(x,y,zw_0)^{\infty}$.
     \item $\mathcal{L} : (x,y,zw_0) =\mathcal{L} + I_3(B(\varphi)).$
     \item $\mathcal{A}=\mathcal{L}:(x,y,zw_0)^2$.
 \end{enumerate}   
\end{theorem}

\begin{notation}
Recall the \Cref{notation 5.1}, $\mathcal{J}=(l_1,\dots,l_{n-2}):(x,y,zw_0)$ and  $B=S/{\mathcal{J}}$ is Cohen-Macaulay domain of dimension 5. Define the $B$-ideal, $\mathcal{K'}$ such that
\begin{align*}
\mathcal{K'}=\overline{(z^2w_0)}:\overline{(x,y,zw_0)}^{2}.
\end{align*}
\end{notation}

From the previous section, $\varphi'$ is strictly equivalent to block diagonal matrix $[L_{n-2}']$ or $[L_r'~M_{r'}]$. Depending on wether $w_0 \in I_2(B(\varphi''))$ or $w_0 \notin I_2(B(\varphi''))$,
we give the complete characterization of the $B$-ideal,
${\mathcal{K}}'$.

\begin{theorem}\label{subcase five}
    With the assumption of \Cref{Main Setting}.
    \begin{enumerate}
        \item Let $\varphi'$ is strictly equivalent to $[L'_{n-2}]$.
        \begin{enumerate}
            \item If $w_0 \notin I_2(B(\varphi''))$, then the $B$-ideal,
            \begin{align*}\mathcal{K'}=(\overline{z^2w_0})+(\{ \overline{z\a_i} ; 1 \leq i \leq n-3 \}) \\ +(\{\overline{-w_{n-2}\a_{j+1}+w_{n-1}\a_j} ; 1\leq j \leq {n-4}\}).\end{align*}
            \item If $w_0 \in I_2(B(\varphi''))$, then the $B$-ideal,
            \begin{align*}
            \mathcal{K'}=(\overline{z^2w_0})+(\{\overline{z\a_i} ; 1 \leq i \leq n-3\})+(\{ \overline{\a_i \a_{i'}} ; 1 \leq i,i' \leq n-3 \}) \\+(\{{\overline{w_{n-1}(\a_{i+1}-\a_i)-w_{n-2}(\a_{i+2}-\a_{i+1})}} ; 1 \leq i \leq n-5\}).\end{align*}
        \end{enumerate}
        \item Let $\varphi'$ be strictly equivalent to $[L'_r~M_{r'}]$ and $r' \geq 2$.
        \begin{enumerate}
            \item Then $w_0$ is always a zero divisor over $I_2(B(\varphi''))$.
            \item The $B$-ideal,
            \begin{align*}
                \mathcal{K'}=(\overline{z^2w_0}) + (\{ \overline{z\a_i} ; 1 \leq i \leq n-3\}) + (\{ \overline{\a_i \a_j} ; 1 \leq i,j \leq n-3\}) \\ +((\overline{dw_{n-2}+bw_{n-1})\a_1-(cw_{n-2}+aw_{n-1})\a_2}) \\ + (\overline{(cw_{n-2}+aw_{n-1})\a_3-(dw_{n-2}+bw_{n-1})\a_2+(cw_{n-2}+aw_{n-1})\a_1}) \\ + (\overline{(dw_{n-2}+bw_{n-1})\a_{i+1}-(cw_{n-2}+aw_{n-1})\a_i ; 3 \leq i \leq n-5}), \end{align*}
         where $a,b,c,d \in \k$ such that the $(n-3)$ symmetric equation $l_{n-3}=zw_0+(aw_{n-1}+cw_{n-2})x+(bw_{n-1}+dw_{n-2}y)$.   
        \end{enumerate}
        \item Let $\varphi'$ be strictly equivalent to $[L'_{n-3}~M_1]$, where $M_1=[ax+by]; a,b \in \k$.
        \begin{enumerate}
            \item If $w_0 \notin I_2(B(\varphi'')$, then the $B$-ideal,
            \begin{align*}
                \mathcal{K'}=(\overline{z^2w_0})+(\{\overline{z \a_i} ; 1 \leq i \leq n-3\})+(\{\overline{-aw_{n-1}\a_{j+1}+bw_{n-1}\a_j} ; 1\leq j \leq {n-4}\}).
            \end{align*}
            \item If $w_0 \in I_2(B(\varphi''))$, then the $B$-ideal,
            \begin{align*}
                \mathcal{K'}=(\overline{z^2w_0})+(\{ \overline{z \a_i} ; 1 \leq i \leq n-3 \})+(\{ \overline{\a_i \a_{i'}} ; 1 \leq i,i' \leq n-3\})\\+(\{{\overline{bw_{n-1}(\a_{i+1}-\a_i)-aw_{n-1}(\a_{i+2}-\a_{i+1})}} ; 1 \leq i \leq n-5\}).
            \end{align*}
        \end{enumerate}
    \end{enumerate}
\end{theorem}
\begin{proof}
    Using the Grobner basis technique and tracing the proof of \Cref{w_0 is regular over ideal}, we have $\overline{x}$  regular over the candidate (say $N$) for $\mathcal{K'}$ in each case. Thus, we have $N_p=\mathcal{K'}_p$ for all $p \in \Ass(N)$. Hence $N=\mathcal{K'}$.
\end{proof}

Define  $\mathcal{D}'= \frac {\overline{l_{n-1}}\overline{\mathcal{K'}}}{\overline{z^2w_0}}$, then 
$\mathcal{D}'$ is well defined $B$-ideal (See for example \cite[Proposition 6.5]{SurajMukundan24}). Also $\overline{\mathcal{A}}$ is prime ideal and $\overline{l_{n-1}}\in \overline{\mathcal{A}}$ but $\overline{z},\overline{w_0}\notin \overline{\mathcal{A}}$ therefore $\mathcal{D}' \subseteq \overline{\mathcal{A}}$.

\begin{theorem}\label{equality of D' and A in case iii}
Assume the \Cref{Main Setting}, the $B$-ideals $\mathcal{D}'$ and $\mathcal{\overline{A}}$ are equal.    
\end{theorem}
\begin{proof}
Proof is similar to the proof of \Cref{another characterization of defining ideal of Rees algebra}.    
\end{proof}

\section{Description of Defining ideal of Rees algebra when presentation matrix is in Case III}\label{section 6 and rees algebra in case iii}
Assume the \Cref{Main Setting}, and that the presentation matrix $\varphi$ is of Case III (see \Cref{case i and ii presentation matrix}).  %{In this section we will prove $I$ is not of linear type.}
In this section we will show there exist some non-zero element in $\mathcal{A}$ which does not belongs to $\mathcal{L}$ and that this element can be computed from the Jacobian dual matrix. Also we will compute the defining ideal $\mathcal{A}$ of Rees algebra $\rees(I)$ explicitly and prove that it is Cohen-Macaulay.

\begin{notation}
Consider the $S$-ideal, $(x,y,z^2w_0)$ then $\underline{w} \cdot \varphi = (x,y,z^2w_0) \cdot B(\varphi)$. Here $B(\varphi)$ is Jacobian dual of presentation matrix $\varphi$ with respect to ideal $(x,y,z^2w_0)$. Throughout this section we will follow this definition of Jacobian dual.

Let $J=(l_1,\dots,l_{n-2}) \subseteq \mathcal{L}$ and $\varphi''$ be $n \times (n-2)$ submatrix of $\varphi$, obtained after deleting the last column of $\varphi$. From presentation matrix $\varphi$ (Case III of \Cref{case i and ii presentation matrix}), we have $J \subseteq (x,y)$ thus $[l_1~\cdots~l_{n-2}]=[x~y] \cdot B(\varphi'')$. Define the $S$-ideal, $\mathcal{J}=J:(x,y)$ and consider the quotient ring $B=S/{\mathcal{J}}$. Let $f \in S$ and  $\overline{f}$ be image of $f$ in the ring $B$. Define the $B$-ideal,
\begin{align*}
    \overline{\mathcal{K}}'=(\overline{z^2w_0}) : (\overline{x,y,z^2w_0}).
\end{align*}
\end{notation}

\begin{observation}
    With the assumption of \Cref{Main Setting}, the $S$-ideal, $\mathcal{J}$ is Cohen-Macaulay prime ideal of height $n-2$ and 
    \begin{align*}
        \mathcal{J}=J:(x,y)^{\infty}=J+I_2(B(\phi'')).
    \end{align*}
    Furthermore, the ideal $I_2(B(\phi''))$ is prime ideal of height $n-3$.
\end{observation}
\begin{proof}
 From \Cref{colon ideal related to defining ideal of Rees algebra}, we have $\mathcal{A}=\mathcal{L}:(x,y)^{\infty}$. Thus \cite[Proposition 2.2, Theorem 2.3] {BoswellMukundan16} imply 
 \begin{align*}
     \mathcal{J}=J:(x,y)^{\infty}=J:(x,y)=J+I_2(B(\varphi'')),
 \end{align*}
is Cohen-Macaulay prime ideal of height $n-2$. And, hence $I_2(B(\varphi''))$ is prime ideal of height $n-3$.
\end{proof}

\begin{observation}
    The ring $B$ is Cohen-Macaulay domain of dimension $5$.
\end{observation}
\begin{proof}
    Since $\mathcal{J}$ is Cohen-Macaulay prime ideal of the height $n-2$, therefore $B$ is Cohen-Macaulay domain of dimension 5.
\end{proof}

\begin{theorem}\label{K in case III}
    With the assumption of \Cref{Main Setting}. The $B$-ideal,
    \begin{align*}
        \overline{\mathcal{K}}''=(\overline{z^2w_0})
        :(\overline{x,y,z^2w_0})=(\overline{z^2w_0}).
    \end{align*}
    Also, $\overline{\mathcal{K}}''$ is Cohen-Macaulay ideal of height one.
\end{theorem}
\begin{proof}
    The ring $B$ is Cohen-Macaulay domain. Thus the $B$-ideal, $(\overline{z^2w_0})$ is Cohen-Macaulay ideal of height one.
    Now we will prove that the $B$-ideal, $(\overline{x,y,z^2w_0})$ is of height two, that is the $S$-ideal, $\mathcal{J}+(x,y,z^2w_0)$ is of height $n$. Let $Q$ be a minimal prime of $\mathcal{J}+(x,y,z^2w_0)$ then $\mathcal{J}+(x,y,z^2w_0) \subseteq Q$, that is, $(x,y,z^2w_0) +I_2(B(\varphi'')) \subseteq Q$. Thus either $(x,y,z)+I_2(B(\varphi'')) \subseteq Q$ or $(x,y,w_0)+I_2(B(\varphi'')) \subseteq Q$. Since $I_2(B(\varphi'')) \subseteq \k[\underline{w}]$ is prime ideal of of height $n-3$ therefore $(x,y,z)+I_2(B(\varphi''))$ and $(x,y,w_0)+I_2(B(\varphi''))$ is of height $n$. Thus $\mathcal{J}+(x,y,z^2w_0)$ is of height $n$ and hence $(\overline{x,y,z^2w_0})$ is of height two.

    Now for any $P \in \Ass{(\overline{z^2w_0})}$, $\hgt P=1$ and this implies $P \notin \Ass{(\overline{x,y,z^2w_0})}$. Thus $(\overline{x,y,z^2w_0})_P=B_P$, for all $P \in \Ass{(\overline{z^2w_0})}$. So, for all $P \in \Ass{(\overline{z^2w_0})}$ we have,
    \begin{align*}
        (\overline{z^2w_0})_P \subseteq (\overline{z^2w_0})_P : (\overline{x,y,z^2w_0})_P=(\overline{z^2w_0})_P:B_P=(\overline{z^2w_0})_P.
    \end{align*}
    Hence, $\overline{\mathcal{K}}''=(\overline{z^2w_0})$ is Cohen-Macaulay ideal of height one.
\end{proof}
    Consider $\mathcal{D}=\frac{\overline{l_{n-1}}{\overline{\mathcal{K}}''}}{\overline{z^2w_0}}$, \Cref{K in case III} implies, $\mathcal{D}$ is well defined $B$-ideal. Also $\overline{l_{n-1}} \in \overline{\mathcal{A}}$, so $\overline{l_{n-1}{\mathcal{K}}}'' \subseteq \overline{\mathcal{A}}$. That is, $\overline{z^2w_0}{\mathcal{D}} \subseteq \overline{\mathcal{A}}$. Since $\overline{\mathcal{A}}$ is the prime ideal and $\overline{z^2w_0} \notin \overline{\mathcal{A}}$ therefore, $\mathcal{D} \subseteq \overline{\mathcal{A}}$.

 \begin{theorem}\label{equality of D and A in case iii}
     With the assumption of \Cref{Main Setting}. The $B$-ideal, $\mathcal{D}$ and $\overline{\mathcal{A}}$ are equal.
 \end{theorem}
\begin{proof}
 Proof is similar to the proof of \Cref{another characterization of defining ideal of Rees algebra}.   
\end{proof}

\begin{theorem}\label{explicit expression for defining ideal of Rees algebra}
 Assume the \Cref{Main Setting}, then the defining ideal of the Rees algebra, $$\mathcal{A}= \mathcal{L}+ I_3(B(\varphi)).$$
\end{theorem}
\begin{proof}
Follows from \Cref{K in case III}, \Cref{equality of D and A in case iii} and the fact that $I_2(B(\varphi''))=I_3(B(\varphi))$.
\end{proof}

\begin{cor}
    The Rees algebra $\mathcal{R}(I)$ is Cohen-Macaulay.
\end{cor}

\section{examples}\label{section 7 and examples}
In this section, we provide some examples to illustrate for \Cref{presentation matrix},  \Cref{minimal prime is unique}. In all of the following examples, one can use Macaulay2 (\cite{M2} )to check that the ideals $I$ satisfy $\Gs{{2}}$ but not $\Gs{3}$ by using \Cref{Gs and Fitting ideals}.

In the following example, the presentation matrix $\varphi$ is in Case I, and we will compute the defining ideal of the Rees algebra.

\begin{example}\label{example 1}
Let $R=\k[x,y,z]$ and consider the matrix 
\begin{align*}
    \varphi=\begin{bmatrix}
    x & y & z & xz\\
    x & 0 & 0 & x^2\\
    y & x & 0 & y^2\\
    0 & y & x & xy\\
    0 & 0 & y & y^2
    \end{bmatrix}
\end{align*}
Here $n=5$, so we have $I=I_4(\varphi)$. %It is clear that $Min(I_3(\varphi))=\{\}$, thus $\hgt I_3(\varphi) =2$. 
The Hilbert-Burch Theorem (see for example \cite[Theorem 20.15]{eisenbud_book}) implies that $I=I_4(\varphi)$ is a height 2 perfect ideal. Also, $\Min(I_3(\varphi))={(x,y)}$ and hence $\hgt I_3(\varphi)=2$.

Let $\mathcal{J}=(l_1,l_2,l_3):(x,y,zw_0)=(l_1,l_2,l_3,w_0^2+w_0w_1-w_2^2+w_0w_3+w_1w_3)$. Define $B=S/{\mathcal{J}}$, then $\overline{\mathcal{K}}^{(2)}=\overline{(z^2w_0^2)}:\overline{(x,y,zw_0)}^{(2)}=\overline{(z^2w_0^2,w_4\a_1-w_3\a_2)}$, where $\a_1=w_0w_4+w_1w_4-w_2w_3, \a_2=w_2w_4-w_0w_3-w_2w_3$. Thus, $\overline{\mathcal{A}}=\frac{\overline{l_4 {\mathcal{K}}}^{(2)}}{\overline{z^2w_0^2}}$.
\end{example}

The following example satisfies the setting of \Cref{subcase five}. The Rees algebra of the ideal in this example is not Cohen-Macaulay.

\begin{example}\label{example 2}
Let $R=\k[x,y,z]$ and consider the matrix 
\begin{align*}
    \varphi=\begin{bmatrix}
    x & y & z & x^2\\
    x & 0 & 0 & xz\\
    y & x & 0 & y^2\\
    0 & y & x & xy\\
    0 & 0 & y & y^2
    \end{bmatrix}
\end{align*}
Here $n=5$, so we have $I=I_4(\varphi)$. %It is clear that $Min(I_3(\varphi))=\{\}$, thus $\hgt I_3(\varphi) = 2$. 
The Hilbert-Burch Theorem (see for example \cite[Theorem 20.15]{eisenbud_book}) implies that $I=I_4(\varphi)$ is a height 2 perfect ideal. Also, $\Min(I_3(\varphi))={(x,y)}$ and hence $\hgt I_3(\varphi)=2$.

Let $\mathcal{J}=(l_1,l_2,l_3):(x,y,zw_0)=(l_1,l_2,l_3,w_0^2+w_0w_1-w_2^2+w_0w_3+w_1w_3)$. Define, $B=S/{\mathcal{J}}$, then $\overline{\mathcal{K}}'=\overline{(z^2w_0)}:\overline{(x,y,zw_0)}^{(2)}=\overline{(z^2w_0,z\a_1,z\a_2,w_4\a_1-w_3\a_2)}$, where $\a_1=w_0w_4+w_1w_4-w_2w_3, \a_2=w_2w_4-w_0w_3-w_2w_3$. Thus, $\overline{\mathcal{A}}=\frac{\overline{l_4 {\mathcal{K}}}'}{\overline{z^2w_0}}$.
\end{example}

In the following example, the ideal $I$ satisfies the setting of \Cref{K in case III}. The Rees algebra of ideals in this case is Cohen-Macaulay.

\begin{example}\label{example 3}
{\rm
Let $R=\k[x,y,z]$ and consider the matrix 
\begin{align*}
    \varphi=\begin{bmatrix}
    x & x & 0 & z^2\\
    x & 0 & 0 & x^2\\
    y & x & 0 & y^2\\
    0 & y & x & yx\\
    0 & 0 & y & y^2
    \end{bmatrix}
\end{align*}
Here $n=5$, so we have  $I=I_4(\varphi)$. %It is clear that $\Min(I_3(\varphi))=\{\}$, thus $\hgt I_3(\varphi) = 2$. 
The Hilbert-Burch Theorem (see for example \cite[Theorem 20.15]{eisenbud_book}) implies that $I=I_4(\varphi)$ is a height 2 perfect ideal. Also, $\Min(I_3(\varphi))={(x,y)}$ and hence $\hgt I_3(\varphi)=2$. 
Defining ideal $\mathcal{L}$ of Symmetric algebra $\Sym(I)$ is 
\begin{align*}
    \mathcal{L}=(xw_2+(-x-y)w_3+yw_4,xw_1-yw_2,xw_0-yw_1+xw_3-yw_4,z^2w_0-z^2w_1+x^2w_3+yxw_4)
    \end{align*}
    and Jacobian dual of matrix $\varphi$ is
\begin{align*}
B(\varphi)= \begin{pmatrix}
w_0+w_1&w_0+w_2&w_3& xw_1 \\ 
w_2&w_3&w_4&yw_2+xw_3+yw_4\\ 
0 & 0 &0&1 
\end{pmatrix}
\end{align*}
then $I_3(B(\varphi))=(-w_0w_2-w_2^2+w_0w_3+w_1w_3,-w_2w_3+w_0w_4+w_1w_4,-w_3^2+w_0w_4+w_2w_4)$ so generators of $I_3(B(\varphi))$ are in  $\k[w_0,w_1,w_2,w_3,w_4]$ is of degree 2. Therefore defining ideal $\mathcal{A}$ of Rees algebra $\rees(I)$ is 
\begin{align*}
    \mathcal{A}= \mathcal{L}:(x,y)&= \mathcal{L} +I_3(B(\varphi)).
\end{align*}
}
\end{example}

\subsection*{Acknowledgements}
The author would like to thank his advisor, Prof. Vivek Mukundan, for several helpful conversations and insightful comments.

The author is also thankful to the makers of the computer algebra software Macaulay2 \cite{M2}, which helped to verify the examples. 

\bibliographystyle{plain}
\bibliography{refrences}

\end{document}